\documentclass[11pt]{article}
\usepackage[utf8]{inputenc}
\usepackage[T1]{fontenc}
\usepackage{latexsym,amssymb,amsmath,amsfonts,amsthm}
\usepackage{graphics}
\usepackage[section]{placeins}
\usepackage{graphicx}
\usepackage{mathrsfs}
\usepackage{subfigure}
\usepackage{float}
\usepackage{color}
\usepackage{mathtools}
\usepackage{tikz}
\usetikzlibrary{calc}
\usepackage{epstopdf}
\newcommand{\R}{{\mathbb R}}

\newcommand{\be}{\begin{eqnarray}}
\newcommand{\ben}{\begin{eqnarray*}}
\newcommand{\en}{\end{eqnarray}}
\newcommand{\enn}{\end{eqnarray*}}

\newcommand{\pa}{\partial}

\newcommand{\real}{{\rm Re\,}}
\newcommand{\ima}{{\rm Im\,}}

\newcommand{\s}{\mathbb{S}}

\newcommand{\range}{{\rm Range}}

\newcommand{\supp}{\mbox{supp}}

\newtheorem{theorem}{Theorem}[section]

\newtheorem{lemma}[theorem]{Lemma}
\newtheorem{corollary}[theorem]{Corollary}

\newtheorem{assumption}{Assumption}[section]

\definecolor{mgq}{rgb}{0,0,0}
\definecolor{ghx}{rgb}{0,0,0}
\definecolor{rot1}{rgb}{0,0,0}
\definecolor{rot}{rgb}{0,0,0}
\definecolor{rot2}{rgb}{0,0,0}
\definecolor{mgq1}{rgb}{0,0,0}
\definecolor{mgq2}{rgb}{0,0,0}
\newcommand{\mgq}{\color{mgq}}
\newcommand{\ghx}{\color{ghx}}

\newcommand{\mgqa}{\color{mgq1}}
\newcommand{\mgqb}{\color{mgq2}}

\begin{document}
\renewcommand{\theequation}{\arabic{section}.\arabic{equation}}
\begin{titlepage}
\title{\bf  A frequency-domain method to inverse moving source problem with unknown radiating moment}

\author{Guanqiu Ma\thanks{School of Mathematical Sciences, Sichuan Normal University, 610066 Chengdu, China. ({\tt gqma@sicnu.edu.cn})} \and
Hongxia Guo\thanks{Corresponding author: School of Mathematical Sciences and Institute of Mathematics and Interdisciplinary Sciences, Tianjin Normal  University, 300384, Tianjin, China. ({\tt hxguo@tjnu.edu.cn})}  \and Guanghui Hu\thanks{School of Mathematical Sciences and LPMC, Nankai University, 300071 Tianjin, China. ({\tt ghhu@nankai.edu.cn})}
}
\date{}
\end{titlepage}
\maketitle

%\vspace{.2in}

\begin{abstract}

This paper introduces a multi-frequency factorization method for imaging a time-dependent source, specifically to recover its spatial support and the associated excitation instants.
 Subject to appropriate assumptions on the source velocity and wave speed, the proposed method is capable of reconstructing the orbit and shape of a moving extended source.
Using far-field data from two opposite directions, we establish a computational criterion that characterizes both the unknown pulse moments and the narrowest strip (perpendicular to the direction) enclosing the source support. Central to our inversion scheme is the construction of indicator functions, defined pointwisely over the spatial and temporal sampling variables.
The proposed inversion scheme permits the recovery of the $\Theta$-convex support domain from far-field data at sparse observation directions. Uniqueness in determining the convex hull of the support and the excitation instants-using all observation directions-is also established as a direct consequence of the factorization method. The effectiveness and feasibility of the approach are examined through comprehensive numerical simulations in two and three  dimensions.

\vspace{.2in} {\bf Keywords: Inverse source problem, Helmholtz equation, time-dependent sources, moving extend source, multi-frequency far-field data, factorization method.
}
\end{abstract}

\section{Introduction}
\subsection{Problem formulation in the time domain}
Assume that the space $\mathbb{R}^3$ is filled by a homogeneous and isotropic acoustic medium. We designate the sound speed in the background medium as the constant $c > 0$.  Consider the acoustic radiating problem incited by a moving extended source with the time-varying source profile $D(t)\subset \R^3$ ($t>0$), which emits wave signals at discrete time moments $t_j\in[0,T]$, $j=0,1,2,\cdots J$ for some $T>0$.
This source traces a trajectory defined by the continuous function $a(t): [0, T] \rightarrow \mathbb{R}^3$. For each $t_j>0$,
we suppose that $D_j:=D(t_j) \subset \R^3$ is a bounded Lipschitz domain such that $\R \backslash  \overline{D}_j$ is connected. %We denote that  $D_j = \{x\in \R^3: x = y+a(t_j) \mbox{ for some } y\in { \mgq D}\}$ the source support at $t=t_j$.
Then
the propagation of the radiated wave fields $U(x,t)$ is governed by the initial value problem
\begin{equation}\label{timeeqn}
\left\{
\begin{aligned}
&c^{-2}\frac{\partial^2 U}{\partial t^2} = \Delta U + \tilde{S}(x,t), \quad &&(x, t) \in \mathbb{R}^3 \times \mathbb{R}_+,\quad \mathbb{R}_+ \coloneqq \{t\in \R: t>0\},\\
&U(x,0)=\partial_t U(x,0) = 0, &&x\in \mathbb{R}^3.
\end{aligned}
\right.
\end{equation}
%The time-dependent source function $\tilde{S}(x,t)$ is supposed to radiate wave signal at finite unknown moments $\{t_{j}, j=0,1,2,\cdots,J\} \subset [0, T]$.
where the source function $\tilde{S}(x,t)$ is of the form
\begin{equation}
\tilde{S}(x,t) = \sum\limits_{j=0}^{J} S_j(x-a(t))\;\delta(t-t_j).
\end{equation}
Here $\delta$ denotes the Dirac delta function. We are interested in the inverse problem of recovering $D=D_j$ and the excitation instants $t_j$ from wave signals recorded at far fields.

In this work, we make the following assumptions on the trajectory function $a$, the source profile functions $S_j$ and the impulsive moments $t_j>0$.

%For the separability and sequence-preserving property of the excitation signal, $S(x)$ and satisfy the following assumptions:

\begin{assumption}\label{sump}
\,
\begin{itemize}
	{\mgq \item  $a(t)\in C([0,T])$ is a real-valued function, and $|a'(t)|<c, t\in [0,T]$.
	\item  $t_j-t_{j-1}>{\rm{diam}} (D_{j-1}) /c,\, j=1,2,\cdots,J$. Here ${\rm{diam}}(D) := \sup_{x,y\in D}|x-y|$ denotes the diameter of the domain $D$.
	\item  $S_j(x)\in C(\overline{D}_j)$ is a real-valued function, and ${\rm supp} (S_j)=\overline{D}_j \subset B_R$ with some $R>0$. Here $B_R= \{x\in \R^3:|x|<R\}$.
	\item  $|S_j(x)|>0$, a.e. $x\in \overline{D}_j$.}
\end{itemize}
\end{assumption}
According to the first assumption, the extended source travels slower than the wave speed. The second assumption ensures that wave signals generated at different instants remain separable due to non-overlapping supports. This can be rigourously derived from the explicit expression of $U(x, t)$ in terms of the convolution of $\tilde{S}$ and the fundamental solution $G$.
{\ghx Below we will show a simple example to illustrate the separability of wave signals.
Suppose that the extended source is moving along the trajectory $a(t)=(-\frac{1}{2}t, 0, 0)$,  the signal excitation instants are $t_j= j, j= 0,1,...,6$,  and the source function is
$$\tilde{S}(x,t) = 1000 \sum\limits_{j=0}^{6}
\frac{e^{-\frac{|x-a(t)|^2}{2\eta}}}{\sqrt{2\pi}\eta}  \cos t \; \delta(t-t_j), \;\eta=0.01.$$
Here we have supposed that the source profile does not change along the time.
We set the position of the receiver at $x_0=(3, 0, 0)$ and the wave speed $c=1$. Then
$U(x,t)$ can be expressed as
\begin{equation}
U(x,t)=\tilde{S}(x,t)*G(x,t)=  \sum\limits_{j=0}^{J} \int_{|x-y|=t-t_j} \frac{\tilde{S}(y-a(t_j))}{4\pi(t-t_j)}\, ds(y),
\end{equation}
where $*$ denotes convolution between $G$ and $\tilde S$ with respect to both $t$ and $x$, and
\begin{equation}
G(x,t)=\frac{\delta(t-|x|)}{4\pi |x|},\quad t\neq |x|,\quad |x|\neq 0,
\end{equation}
is the Green's function of the wave operator $\pa^2_t-\Delta$ in $\R^3\times\R_+$.
Then we show the radiated waves $U(x_0,t)$  VS time $t$ in Figure \ref{fig:a}.
It is obvious that the radiating signals is separable with respect to time $t$. The arrival moments $t=3, 4.5, 6, 7.5, 9, 10.5, 12$ can be predicted by the travelling time between the location of emitters and receiver.}

\begin{figure}[H]
\centering
\includegraphics[scale=0.4]{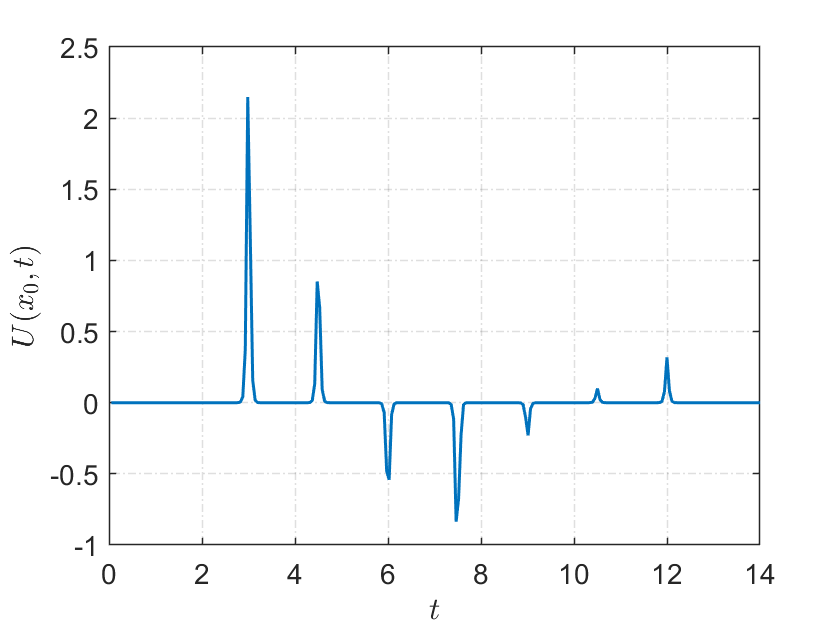} %  0-14
\caption{
Radiating signals $U(x_0,t)$ versus time $t$ at a fixed observation position $x_0$.}\label{fig:a}
\end{figure}

\subsection{Problem reduction in the frequency domain}

Motivated by the separability of wave signals, we reconsider the wave equation \eqref{timeeqn} but with a single excitation instant $t_0$:
\begin{equation}\label{timeeqnsim}
\left\{
\begin{aligned}
&c^{-2}\frac{\partial^2 U}{\partial t^2} = \Delta U + S_0(x-a(t))\; \delta(t-t_0), \quad &&(x, t) \in \mathbb{R}^3 \times \mathbb{R}_+,\\
&U(x,0)=\partial_t U(x,0) = 0, &&x\in \mathbb{R}^3.
\end{aligned}
\right.
\end{equation}
The expression of $U$ takes the form
\begin{equation}
U(x,t)= \int_{|x-y|=t-t_0} \frac{S_0(y-a(t_0))}{4\pi(t-t_0)}\, ds(y),\quad t>0, x\in \R^3.
\end{equation}
Our idea is to investigate the inverse problem in the frequency domain by applying the Fourier transform. {\color{mgq}In this paper the one-dimensional Fourier and inverse Fourier transforms are defined by
\begin{equation*}
(\mathcal{F}f)(k)=\frac{1}{\sqrt{2\pi}}\int_{\R}f(t)e^{-{ i} k t}\,dt,\quad
(\mathcal{F}^{-1}v)(t)=\frac{1}{\sqrt{2\pi}}\int_{\R}v(k)e^{{ i} k t}\,dk,
\end{equation*}
respectively.} 	
The inverse Fourier transform of $S_0(x-a(t))\delta(t-t_0)$ is given by
\begin{equation}\label{sourcef}
f_0(x,\omega) :=\frac{1}{\sqrt{2\pi}}\int_{\mathbb{R}} S_0(x-a(t))\delta(t-t_0) e^{i\omega t}\, dt = \frac{1}{\sqrt{2\pi}} S_0(x-a(t_0)) e^{i\omega t_0}.
\end{equation}
The support of the source function is \supp $f_0=\overline{D_{0}}$. For notational convenience we drop the subscript $j=0$ to rewrite
 $f_0$, $S_0(x-a(t_0))$ and $D_0$ as $f$, $S(x)$ and $D$, respectively.

Taking the inverse Fourier transform on the wave equation \eqref{timeeqnsim} yields the inhomogeneous Helmholtz equation
\begin{equation}\label{eq1}
\Delta u(x,\omega) + \frac{\omega^2}{c^2}  u(x,\omega) = -f(x,\omega), \qquad x\in \R^{3}, \;\omega>0.
\end{equation}
At infinity $u$ should be complemented with 	
the Sommerfeld radiation condition
\be\label{SRC}
\lim\limits_{r \to \infty} r (\partial_r u - i\frac{\omega}{c}u) = 0,\quad r = |x|,\en
which holds uniformly in all directions $x/|x|$.
%Denote by $[\omega_{\min}, \omega_{\max}]$ an interval of frequencies on the positive real axis.
For every $\omega > 0$, the solution $u\in H^2_{loc}(\R^3)$ of the equation \eqref{eq1} is given by
\begin{equation}\label{expression-w}
u(x, \omega) = \int_{{D}} \Phi(x-y;\omega /c) f(y, \omega) dy,  \quad x \in \R^3 .
\end{equation}
Here, $\Phi(x;k)$ is the fundamental solution to the Helmholtz equation $(\Delta + k^2)u = 0$, given by
\begin{equation*}
\Phi(x;k) = \frac{e^{ik|x|}}{4 \pi |x|},\quad x\in \mathbb{R}^3,\, r=|x| \neq 0.
\end{equation*}
The Sommerfeld radiation condition \eqref{SRC}  gives rise to the following asymptotic behavior at infinity:
\begin{equation}\label{far-field}
u(x,\omega)=\frac{e^{{ i} \omega c^{-1}|x|}}{4\pi|x|}\left\{u^\infty(\hat{x},\omega)+O\left(\frac{1}{r}\right)\right\}\quad\mbox{as}\quad|x|\rightarrow\infty,
\end{equation}
where $u^\infty(\cdot, \omega)\in C^\infty(\s^2)$ is known as the far-field pattern (or scattering amplitude) of $u$. It is also well known that the function  $\hat{x}\mapsto u^\infty(\hat{x}, \omega)$ is real analytic on $\s^2$, where $\hat{x}\in \s^2$ is usually referred to as the observation direction.
By \eqref{expression-w}, the far-field pattern $u^\infty$ of $u$ can be expressed as
\be\label{u-infty}
u^\infty(\hat{x}, \omega)=\int_{{D}} e^{-{ i}\omega c^{-1}\hat{x}\cdot y} f(y, \omega)\,dy = \frac{1}{\sqrt{2\pi}}\int_{{D}} e^{-{ i}\omega (c^{-1}\hat{x}\cdot y -t_0)} S(y)\,dy ,\, \hat{x}\in \s^2,\, k>0.
\en
Noting that the function $S$ is real valued, we have $f(x, -\omega)=\overline{f(x,\omega)}$  and thus $u^\infty(\hat{x}, -\omega)=\overline{u^\infty(\hat{x}, \omega)}$ for all $\omega>0$.

The inverse source problem with a single excitation instant in the frequency domain can be rephrased as
\\ \\
(ISP): Determine the pulse moment $t_0>0$ and the position and shape of the support ${D}$ of $S$ from knowledge of the multi-frequency far-field patterns
$$\left\{u^\infty(\hat{x}_m, \omega): \omega\in[\omega_{\min}, \omega_{\max}], \, m=1,2,\cdots, M.\right\}. $$
It deserves to note that $u^\infty(\hat{x}, \omega)$ is just the inverse Fourier transform of the time-dependent far-field pattern $U^\infty(\hat{x}, t)$ of the wave equation \eqref{timeeqn}, given by
\ben
U^\infty(\hat{x}, t)&=&4\pi \lim_{|x|\rightarrow\infty} |x|\,U(x, t+|x|)\\
&=&\int_{\mathbb{R}^3} S_0(y-a(t+\hat{x}\cdot y))\;\delta(t+\hat{x}\cdot y-t_0)\,dy.
\enn

Similarly, we can state the inverse source problem with a finite number of excitation instants as follows.
\\ \\
(ISP'): Determine the pulse moments $\{t_j,j=0,1,2,\cdots,J\}$ and the positions and shapes of the supports ${D_j}$ of $S_j$ from knowledge of the multi-frequency far-field patterns
$$\left\{u_j^\infty(\hat{x}_m, \omega): \omega\in[\omega_{\min}, \omega_{\max}], \, m=1,2,\cdots, M, j=0,1,2,\cdots,J \right\}. $$
Here $u_j^\infty(\hat{x}, \omega)$ denotes the Fourier transform of the far-field pattern corresponding to \eqref{eq1} with $f(x,\omega)=S_j(x-a(t_j)) e^{i\omega t_j}/\sqrt{2\pi}$.

\subsection{Literature review}

The reconstruction of source support and excitation instants in engineering often depends on travel time analysis, which requires precise arrival-time calculations. The present work proposes a Multi-frequency Factorization Method (MFFM), an approach derived from multi-static inverse scattering problems at a fixed freqeuency \cite{K98,KG08}. Previous research has extensively addressed cases where the source function’s time dependence (e.g., excitation instants) is known, providing uniqueness proofs \cite{BLLT, BLT10, CIL, EV09, LY}, increasing stability analyses \cite{BLLT, BLT10, CIL, EV09, LY}, and numerical schemes like iterative, Fourier, and test-function methods for source recovery \cite{BLLT, BLRX, EV09, ZG}. Sampling-type methods for support imaging have also been developed \cite{AHLS, GS, LMZ, JLZ2019}. Within this context, the MFFM was initially conceived in \cite{GS17} to recover the spatial support of a positive definite source. Subsequently, \cite{GGH2022,GHM2023,MGH2023} extended it to handle frequency-dependent sources using frequency intervals of arbitrary length and to track moving point sources with given periods. Building upon \cite{GS17,GGH2022}, this work extends the MFFM framework to incorporate the simultaneous determination of both excitation instants and spatial support, introducing novel ideas to strengthen its conceptual rigorousness.

In terms of moving sources, various inversion algorithms have been proposed for recovering the orbit, profile and magnitude of a moving point source, such as the algebraic method \cite{NIO2012, T2020, Ohe2011}, the time-reversal method \cite{GF2015}, the method of fundamental solutions \cite{CGMS2020}, matched-filter and correlation-based imaging scheme \cite{FGPT2017}, the iterative thresholding scheme \cite{Liu2021} and the Bayesian inference \cite{LGS2021, WKT2022}. See also \cite{LGY21, HKLZ2019, HKZ2020,HLY20,T2022} for uniqueness and stability results on inverse problems of identifying moving sources.
In prior studies \cite{GHM2023,MGH2023}, the author introduced a factorization method for reconstructing the orbit of moving point sources using sparse far-field/near-field data.
To our knowledge, research on moving extended sources is limited. This paper builds upon the ideas presented in \cite{GHM2}, utilizing multi-frequency far-field data from two opposite observation directions. We propose a computational criterion to identify unknown pulse moments. Reconstructed pulse moments and indicator functions are then employed to reconstruct the position and shape of the extended source. Far-field data from sparse observation directions aid in recovering the $\Theta$-convex domain of the support.

The proposed method has the following key features:
(i) Using multi-frequency far-field measurements from two opposing directions, it provides a necessary and sufficient condition for determining the smallest strip (perpendicular to the observation directions) that contains the support of the source. In this sense, it preserves the core principle of the multi-static factorization method.
(ii) Our frequency-domain approach allows for rigorous analysis and efficient inversion using data from frequency intervals of any length. Therefore, it can be particularly effective when high-frequency or low-frequency data are not available.

The main innovations of this work are outlined below.
(i) Joint reconstruction of the source's spatial support and its excitation instants. Formulated in the frequency domain, this problem involves a specific frequency-dependent source and requires indicator functions that incorporate both spatial and temporal variables.
(ii) Validation of pointwise defined test functions.
 Earlier approaches \cite{GS17,GGH2022} relied on test functions integrated over a small spatial domain. Numerical experiments, however, indicate the advantages of pointwise-defined test functions in terms of conciseness and clarity. Building on this observation and guided by asymptotic factorization theory, we provide a rigorous foundation for pointwise test functions and also derive uniqueness theorems.
(iii) Generalized source assumption. We weaken the positivity condition of \cite{GS17,GGH2022} by allowing the source function to vanish on a set of zero measure.
(iv) Broader implications for nonlinear scattering. Although we solve a linear inverse source problem, our approach provides a new framework for tackling nonlinear inverse scattering under the Born or physical optics approximation, which can be reduced to a linear form amenable to our method.

The remaining part is organized as follows. In Section \ref{sec2}, we introduce the multi-frequency far-field operator and connect its range to the data-to-pattern operator under weaken assumptions on the source function. In Section \ref{sec3}, the test functions are designed for the connection between the sampling points and the range of the data-to-pattern operator.
Section \ref{sec4} is devoted to the design of proper indicator functions for characterizing a strip containing the source support and the excitation moment.
The numerical tests will be reported in the final Section \ref{num}.

Below we introduce some notations to be used throughout this paper. A ball centered at $y\in \R^3$ with the radius $\epsilon>0$ will be denoted as $B_\epsilon(y)$. For brevity we write $B_\epsilon=B_\epsilon(0)$ when the ball is centered at the origin. Unless otherwise stated, we always suppose that $D$ is a bounded domain.
Given $\hat{x}\in \s^2$, define
\be \label{xdd}
\hat{x}\cdot D:=\{t\in \R: t=\hat{x}\cdot y\;\mbox{for some}\; y\in D\}\subset \R,
\en
which represents the projection of the domain $D$ along the direction $\hat{x}$.
Hence, $(\inf(\hat{x}\cdot D),  \sup(\hat{x}\cdot D))$ must be a finite and connected interval on the real axis.  %Obviously, $\hat{x}\cdot B_\epsilon(y)=(\hat{x}\cdot y-\epsilon,\hat{x}\cdot y+\epsilon)$.
In the simple case that $\hat{x}=(1,0)$, it holds that (see Figure \ref{F2-D})
$$
\hat{x}\cdot D=\left( \inf_{x\in D} x_1,\; \sup_{x\in D} x_1\right).
$$
\begin{figure}[H]
	\centering
	\scalebox{0.9}{
		\begin{tikzpicture}
			% grating
			%\draw[line width=2.5cm,color=blue!20] (2.5,-3) -- (2.5,3);
		%	\draw[line width=2.5cm,color=blue!20] (-3,-3) -- (-3,3);
		%	\draw[line width=2.5cm,color=green!20] (-0.25,-3) -- (-0.25,3);
		%	\draw (-0.25,2) node [left] {$K_{D}^{(\hat{x})}$};
		%	\draw (3,2) node [left] {${K}_{D,\eta}^{(\hat{x})}$};
		%	\draw (-2.5,2) node [left] {${K}_{D,\eta}^{(-\hat{x})}$};
			\draw[->] (-3,0) -- (3,0) node[above] {$x_1$} coordinate(x axis);
			\draw[->] (0,-3) -- (0,3) node[right] {$x_2$} coordinate(y axis);
		%	\foreach \x/\xtext in {-4,-2, 2, 4}
		%	\draw[xshift=\x cm] (0pt,1pt) -- (0pt,-1pt) node[below] {$\xtext$};
		%	\foreach \y/\ytext in {-4,-2, 2, 4}
		%	\draw[yshift=\y cm] (1pt,0pt) -- (-1pt,0pt) node[left] {$\ytext$};
			
			\draw[domain = -2:360,very thick][samples = 200] plot({cos(\x)+0.65*cos(2* \x)-0.65},{1.5*sin(\x)});
			
			\draw [very thick, densely dotted] (-1.5,-3) -- (-1.5,3);
			\draw (-1.7,0) node [below] {$a$};
			
			\draw [very thick, densely dotted] (1,-3) -- (1,3);
			\draw (1.2,0) node [below] {$b$};
			
		%	\draw [very thick, densely dotted] (1.25,-3) -- (1.25,3);
		%	\draw [very thick, densely dotted] (3.75,-3) -- (3.75,3);
		%	\draw [very thick, densely dotted] (-4.25,-3) -- (-4.25,3);
		%	\draw [very thick, densely dotted] (-1.75,-3) -- (-1.75,3);
			\draw (-0.5,-0.4) node [below] {$D$};
			
			\draw [thick,->] (2.1,1) -- (2.5,1);
			\draw (2.5,1) node[right] {$\hat{x} = (1,0)$};
			
		\end{tikzpicture}
	}
	\caption{Illustration of the interval $\hat{x}\cdot D=(a,b)$ for $\hat{x}=(1,0)$ in two dimensions. }\label{F2-D}
\end{figure}

\section{Range of far-field operator}\label{sec2}
In this section, we will introduce the multi-frequency far-field operator $F$ for a fixed far-field observation direction $\hat{x}\in \s$ and factorize it into the symmetric form $F=LTL^*$. %The operator $L$ will be referred to as the data-to-pattern operator.
Following the ideas of \cite{GS}, we introduce the central frequency  $\kappa$ and  half of the bandwidth of the given data as
\ben
\kappa:=\frac{\omega_{\min}+\omega_{\max}}{2},\quad K :=\frac{\omega_{\max}-\omega_{\min}}{2}.
\enn
Using these parameters, we can define a far-field operator whose domain and range are both the space $L^2(0, K)$. Moreover, our analysis applies to frequency intervals $(\omega_{\min}, \omega_{\max})$ of any length, even when the interval does not include zero.
For every fixed direction $\hat{x}\in \s^2$, we define the far-field operator by
\begin{equation} \label{def:F}
\begin{aligned}
(F\phi)(\tau)=(F_{{D}}^{(\hat{x})}\phi)(\tau)&:=\int_{0}^{K } u^\infty(\hat{x}, \kappa+\tau-s)\,\phi(s)\,ds\\
&=\int_{0}^{K} \int_{{D}} e^{- i(\kappa+\tau-s) (c^{-1}\hat{x}\cdot y -t_0)} \frac{S(y)}{\sqrt{2\pi}}\,dy \,\phi(s)\,ds.
\end{aligned}
\end{equation}
Since $u^\infty(\hat{x},\omega)$ is analytic with respect to the frequency $\omega\in \R$, the operator
$F: L^2(0, K )\rightarrow L^2(0, K )$ is bounded. %We can similarly define the far-field operator $F_j^{(\hat{x})}$ with the far-field pattern $u_j^{\infty}(\hat{x},k)$ corresponding to other pulse moments $t_j$.
%For brevity we write $F = F_0^{(\hat{x})}$ in the discussion that follows.
%By \cite[Theorem 2.1 and Lemma 2.1]{GGH2022}, t
The operator $F$ can be factorized as follows.
\begin{theorem}\label{Fac-F}
We have $F=L\mathcal{T}L^*$, where $L=L_{D}^{(\hat{x})}: {\mgqb H^{-2}({D})}\rightarrow L^2(0, K )$ is defined by
\be\label{def:L}
(Lu)(\tau)=\int_{D} e^{- i \tau (c^{-1}\hat{x}\cdot y -t_0)} u(y)\,dy,\qquad \tau\in [0, K]
\en
for all $u\in {\mgqb H^{-2}({D})}$,
and {\mgqb $\mathcal{T}:  H^2(D)\rightarrow H^{-2}(D)$ is defined as the composition $E\circ M$, where $M: H^2(D)\rightarrow L^{2}(D)$ is the multiplication operator
$$(Mu)(y):=\frac{S(y)}{{\sqrt{2\pi}}} e^{- i \kappa (c^{-1}\hat{x}\cdot y -t_0)}\,u(y), 
$$
and $E:L^2(D)\hookrightarrow H^{-2}(D)$ is the canonical embedding
$$(Eu)(y):=u(y).
$$
Equivalently, for all $u\in  H^{2}({D})$,
\be\label{oT}
(\mathcal{T}u)(y):=E\left(\frac{S(y)}{{\sqrt{2\pi}}} e^{- i \kappa (c^{-1}\hat{x}\cdot y -t_0)}\,u(y)\right). 
\en }
Further, the operator $ {L} : H^{-2}(D) \to L^2(0, K)$ is compact with dense range.
\end{theorem}

{\mgqb The integral kernel of the operator $L$ is an exponential function, which is infinitely smooth. Consequently, the integral operator $L$ maps the space $H^{-2}(D)$ into the space $L^2(0,K)$ when the integration is performed over a Lipschitz domain $D$.}
The operator $L=L_{{\color{mgq}D}}^{(\hat{x})}$ will be referred to as the data-to-pattern operator, because it maps a time-dependent source function to the multi-frequency far-field data at a fixed observation direction, i.e. (see \eqref{u-infty}),
\begin{equation*}
u^{\infty}(\hat{x},\omega) = \frac{1}{\sqrt{2\pi}}(LS)(\omega).
\end{equation*}

% \begin{lemma}\label{com}
% The operator $ {L} : L^2({\color{mgq}D_0}) \to L^2(0, K)$ is compact with dense range.
% \end{lemma}

% \begin{proof}
% For any $\psi \in L^2({\color{mgq}D_0})$, it holds that $ {L}\psi \in H^1(0, K)$, which is compactly embedded into $L^2(0, K)$. This proves the compactness of $ {L}$. By \eqref{aj-tildeL},
% $({L}^* \phi)(y)$ coincides with the inverse Fourier transform of $\phi$ at the variable $c^{-1}\hat{x}\cdot y -t_0$. Since the set $\{c^{-1}\hat{x}\cdot y -t_0: y\in {\color{mgq}D_0}\}$ forms an interval of $\R$,
% the relation $({L}^* \phi)(y)=0$ implies $\phi=0$ in $L^2(0, K)$. Hence, $ {L}^*$ is injective. The denseness of $\text{Range} ( L)$ in $L^2(0, K)$ follows from the injectivity of ${L}^*$.
% \end{proof}
In \cite{GGH2022},
the relationship between the ranges of $F$ and $L$ was established under the strong condition that $|S|$ is strictly positive on $D$, i.e., $|S(x)|\geq c>0$ for all $x\in \overline{D}$. In this work we shall relax the positivity condition to $|S|>0$ a.e. on $\overline{D}$. This relaxed condition can be satisfied as long as the nodal set of $S$ has the Lebesgue measure of zero. Define $F_\# := |\mbox{Re} F|+|\mbox{Im}F|$.

\begin{theorem}\label{RH}
It holds that
\begin{equation}\label{RI}
{\range\,(F_{\#}^{1/2})}={\range\,(L)}.
\end{equation}
\end{theorem}

\begin{proof}
We introduce the nodal set of $S$ and its $\epsilon$-neighborhood by
$$Y:= \{x \in {\overline{D}:}\quad S(x)=0\}, \quad Y_\epsilon:= \bigcup\limits_{y\in Y}(B_\epsilon(y)\cap \overline{D}),$$
 where $\epsilon>0$ is sufficiently small. There holds $D\backslash \overline{Y}_\epsilon \neq \emptyset$ due to the Assumption \ref{sump} on $S$, and meas $\left(\overline{Y}_\epsilon\right) \to 0$ as $\epsilon \to 0$.
Now we define the operator $\mathcal{F}_\epsilon :L^2(0, K )\rightarrow L^2(0, K )$ in the same way as $F$ but over the domain $D\backslash \overline{Y}_\epsilon$ by
\begin{equation} \label{def:Feps}
(\mathcal{F}_\epsilon \phi)(\tau)= \int_{0}^{K} \left[\int_{D\backslash \overline{Y}_\epsilon} e^{- i(\kappa+\tau-s) (c^{-1}\hat{x}\cdot y -t_0)} \frac{S(y)}{\sqrt{2\pi}}\,dy\right] \phi(s)\,ds.
\end{equation}
Analogous to Theorem \ref{Fac-F}, we have  $\mathcal{F}_\epsilon=\mathcal{L}_\epsilon\mathcal{T}_\epsilon\mathcal{L}_\epsilon^*$, where {\mgqb $\mathcal{L}_\epsilon: H^{-2}(D\backslash \overline{Y}_\epsilon)\rightarrow L^2(0, K )$} is defined by
\be\label{def:Leps}
(\mathcal{L}_\epsilon u)(\tau)=\int_{D\backslash \overline{Y}_\epsilon} e^{- i \tau (c^{-1}\hat{x}\cdot y -t)} u(y)\,dy,\qquad \tau\in (0, K)
\en
for all {\mgqb $u\in H^{-2}(D\backslash \overline{Y}_\epsilon)$,
and $\mathcal{T}_\epsilon : H^{2}(D\backslash \overline{Y}_\epsilon)\rightarrow H^{-2}(D\backslash \overline{Y}_\epsilon)$ is defined as the composition $E_\epsilon \circ M_\epsilon$, where $M_\epsilon: H^2(D\backslash \overline{Y}_\epsilon)\rightarrow L^{2}(D\backslash \overline{Y}_\epsilon)$ is the multiplication operator
$$(M_\epsilon u)(y):=\frac{S(y)}{{\sqrt{2\pi}}} e^{- i \kappa (c^{-1}\hat{x}\cdot y -t_0)}\,u(y), 
$$
and $E_\epsilon:L^2(D\backslash \overline{Y}_\epsilon)\hookrightarrow H^{-2}(D\backslash \overline{Y}_\epsilon)$ is the canonical embedding
$$(E_\epsilon u)(y):=u(y).
$$ 
Equivalently, for all $u\in  H^{2}({D\backslash \overline{Y}_\epsilon})$,
\be\label{oTeps}
(\mathcal{T}_\epsilon u)(y):=E_\epsilon \left(\frac{S(y)}{{\sqrt{2\pi}}} e^{- i \kappa (c^{-1}\hat{x}\cdot y -t_0)}\,u(y)\right) .
\en }
For all $v\in H^{2}(D\backslash \overline{Y}_\epsilon)$, 
$$\langle \mathcal{T}_\epsilon u, v \rangle= \frac{1}{\sqrt{2\pi}}\int_{D\backslash \overline{Y}_\epsilon} S(y)e^{- i \kappa (c^{-1}\hat{x}\cdot y -t_0)}\,u(y)\bar{v}(y)\,dy=\langle  u, \mathcal{T}^*_\epsilon v \rangle
$$
where $\big\langle \cdot, \cdot\big\rangle$ represents the dual pairing between $H^{-2}({D\backslash \overline{Y}_\epsilon})$ and $H^{2}({D\backslash \overline{Y}_\epsilon})$ and $\mathcal{T}^*_\epsilon v=E_\epsilon \left(\frac{S(y)}{{\sqrt{2\pi}}} e^{ i \kappa (c^{-1}\hat{x}\cdot y -t_0)}\,v(y)\right): H^{2}(D\backslash \overline{Y}_\epsilon)\rightarrow H^{-2}(D\backslash \overline{Y}_\epsilon)$ is the adjoint operator of $\mathcal{T}_\epsilon$. Hence, 
$$[(\real \mathcal{T}_\epsilon)u](y)=\left(\frac{\mathcal{T}_\epsilon+\mathcal{T}^*_\epsilon}{2}u\right)(y)=E_\epsilon \left(\frac{1}{\sqrt{2\pi}}S(y)\cos (\kappa (c^{-1}\hat{x}\cdot y -t_0))u(y)\right).
$$
Similarly, 
$$[(\ima \mathcal{T}_\epsilon)u](y)=E_\epsilon \left(-\frac{1}{\sqrt{2\pi}}S(y)\sin (\kappa (c^{-1}\hat{x}\cdot y -t_0))u(y)\right).
$$
{\mgqb For any $\epsilon>0$, since $|S(x)| \geq c_\epsilon >0$ for all $x \in \overline{D} \backslash \overline{Y}_\epsilon$, it follows that the operator $\real \mathcal{T}_\epsilon$ is injective from $H^{2}({D\backslash \overline{Y}_\epsilon})$ into $H^{-2}({D\backslash \overline{Y}_\epsilon})$. Moreover, for $\kappa \neq 0$, the operator $\ima \mathcal{T}_\epsilon$ is also injective from $H^{2}({D\backslash \overline{Y}_\epsilon})$ into $H^{-2}({D\backslash \overline{Y}_\epsilon})$. 
Consequently, the operator $\mathcal{T}_{\epsilon,\#}$ is coercive, because the same lower bound $|S(x)| \geq c_\epsilon >0$ for all $x \in \overline{D} \backslash \overline{Y}_\epsilon$ implies the existence of a constant $C>0$ such that 
$$\big\langle \mathcal{T}_{\epsilon,\#}\, u, u\big\rangle\geq c\,||u||_{H^{2}({D\backslash \overline{Y}_\epsilon})}^2\quad\mbox{for all}\quad u\in H^{2}({D\backslash \overline{Y}_\epsilon}),$$
}
%Theorem \ref{Fac-Feps} implies that the far-field operator $\mathcal{F}_\epsilon$ is self-adjoint.
Using the range identity of Theorem A. 1. in the Appendix, {\mgqb that the space $X=H^{2}(D\backslash \overline{Y}_\epsilon), X^*=H^{-2}(D\backslash \overline{Y}_\epsilon)$ and Hilbert space $Y=L^2(0,K)$, }we obtain the relation
\be\label{RIeps}
\mbox{Range}\, (\mathcal{F}_{\epsilon,\#}^{1/2})=\mbox{Range}\,(\mathcal{L}_\epsilon).
\en
{\mgqb 
By the definitions \eqref{def:F} and \eqref{def:Feps}, we obtain $\range \, (\mathcal{F}_\epsilon) \subset \range \, (F)$. Therefore, $\lim\limits_{\epsilon \to 0} \range\, (\mathcal{F}_{\epsilon})\subseteq \range\,(F)$.
Given $h \in \mbox{ Range }(F)$, there exists {\mgqb $\phi \in L^{2}(0,K)$} such that $h=F\phi$. Define $h_\epsilon :=\mathcal{F}_\epsilon(\phi)$. We have
\begin{equation*}
\begin{aligned}
\lim\limits_{\epsilon \to 0}||h- h_\epsilon||^2_{L^2(0,K)} &= \lim\limits_{\epsilon \to 0} \int_{0}^{K} \left| \int_{0}^{K}\int_{D \cap Y_\epsilon} e^{- i(\kappa+\tau-s) (c^{-1}\hat{x}\cdot y -t)} \frac{S(y)}{\sqrt{2\pi}}\,dy\,\phi(s)ds\right|^2 d\tau  \\
& \leq \lim\limits_{\epsilon \to 0} \frac{K}{2\pi} \int_{0}^{K} |\phi(s)|^2\,ds \int_{Y_\epsilon} \left|  S(y) \right|^2\,dy  \\
& \leq  \lim\limits_{\epsilon \to 0}CK||\phi||^2_{L^2(0,K)}|Y_\epsilon|=0.
\end{aligned}
\end{equation*}
It shows that $\range\,(F) \subset {\lim\limits_{\epsilon \to 0}  \range\, (\mathcal{F}_{\epsilon})} $.
Hence, we have
\begin{equation}\label{RIf}
{\lim\limits_{\epsilon \to 0}\range\, (\mathcal{F}_{\epsilon,\#}^{1/2})}={\range\,(F_{\#}^{1/2})}.
\end{equation}

}

{\mgqb Next, we prove the convergence of $\range (\mathcal{L}_{\epsilon})$ and $\range(L)$.}
By the definitions \eqref{def:L} and \eqref{def:Leps}, we can easily obtain Range ($\mathcal{L}_\epsilon$) $\subseteq$ Range ($L$). Therefore, $\lim\limits_{\epsilon \to 0}\mbox{Range}\, (\mathcal{L}_{\epsilon})\subseteq \mbox{Range}\,(L)$.
Given $g \in \mbox{ Range }(L)$, there exists {\mgqb $\psi \in H^{-2}(D)$} such that $g=L\psi$. Define $g_\epsilon :=\mathcal{L}_\epsilon(\psi|_{D\backslash \overline{Y}_\epsilon})$. We have
\begin{equation*}
\begin{aligned}
\lim\limits_{\epsilon \to 0} ||g-g_\epsilon||^2_{L^2(0,K)} &= \lim\limits_{\epsilon \to 0} \int_{0}^{K}\left|  \int_{D \cap Y_\epsilon} e^{- i \tau (c^{-1}\hat{x}\cdot y -t)} \psi(y)\,dy\,\right|^2 d\tau \\
&\leq \lim\limits_{\epsilon \to 0} CK{\mgqb ||\psi||^2_{H^{-2}(D)}}\,|Y_\epsilon| = 0.
\end{aligned}
\end{equation*}
It shows that $\mbox{Range}\,(L) \subset {\lim\limits_{\epsilon \to 0}\mbox{Range}\, (\mathcal{L}_{\epsilon})} $.
Hence, we have
\begin{equation}\label{RIl}
{\lim\limits_{\epsilon \to 0}\mbox{Range}\, (\mathcal{L}_{\epsilon})}={\mbox{Range}\,(L)}.
\end{equation}

Combining \eqref{RIeps}, \eqref{RIf} and \eqref{RIl}, we obtain the range identity
$
{\mbox{Range}\,(F_{\#}^{1/2})}={\mbox{Range}\,(L)}.
$
\end{proof}

According to the proof of \cite[Lemma 3.1]{GGH2022}, we can estimate 
the support of the inverse Fourier transform of $Lu$ as follows. 
\begin{lemma}\label{supplu}For $u\in {\mgqb H^{-2}({D})}$, it holds that
\begin{equation}\label{rl-f}
{\rm supp}\,(\mathcal{F}^{-1}Lu) \subset \left(c^{-1}\inf(\hat{x}\cdot {D})-t_0,\, c^{-1}\sup(\hat{x}\cdot {D})-t_0\right).
\end{equation}
\end{lemma}
Lemma \ref{supplu} will be used in the proof of Theorem \ref{lem3.4} below.
Next we sketch the idea how to extract information on the source support and radiating moment.
Let $\chi_{y,\eta}\in L^2(0, K )$ be a two-parameter-dependent test function with $y\in \R^3$, $\eta\in \R$ and fix the observation direction at $\hat{x}\in \s^2$.
Denote by $(\lambda_n^{(\hat{x})}, \psi_n^{(\hat{x})})$ an eigensystem of the non-negative and self-adjoint operator $ F_{\#}$, which is uniquely determined by the multi-frequency far-field patterns $\{u^\infty(\hat{x}, \omega): \omega \in [\omega_{\min}, \omega_{\max}]\}$. Applying Picard's theorem (Theorem A. 2) and Theorem \ref{RH}, we obtain
\be%\label{indicator}
\chi_{y,\eta} \in {\mgqa {\range\, (L)}}\quad\mbox{if and only if}\quad \sum_{n=1}^\infty\frac{|\langle \chi_{y,\eta}, \psi_n^{(\hat{x})} \rangle|^2}{ |\lambda_n^{(\hat{x})}|}<\infty.
\en
In the subsequent Section \ref{sec3}, we shall choose proper test functions $\chi_{y,\eta}$ such that $\chi_{y,\eta} \in {\mbox{Range}(L)}$ if and only if the temporal sampling variable $\eta$ coincides with the excitation instant $t_0$ and
the spatial sampling variable $z$ belongs to a closed strip related both to the support $D$ and the observation direction $\hat{x}$. Moreover, the instant $t_0$ and the narrowest strip containing the support $D$ can be both recovered from the multi-frequency data of any two opposite directions.

\section{Test functions}\label{sec3}

For any $y\in \R^3$ and $\eta \in \R$, introduce the pointwise defined test functions $\chi_{y,\eta}:=\phi^{(\hat{x})}_{y, \eta}\in L^2(0, K )$  by
\be\label{Test}
\phi^{(\hat{x})}_{y,\eta}(k):= e^{-i k\,  (c^{-1}\hat{x}\cdot y -\eta)},\qquad k\in {\mgq (0, K)}.
\en
Here the point $y\in \R^3$ is the sampling variable in space and $\eta\in \R$ the sampling variable in time. Define the unbounded and parallel strips (see Figure \ref{strip}):
{\mgqa
\be\label{K}
K_{D}^{(\hat{x})}&:=&\{y\in \R^3: \inf(\hat{x}\cdot D) {\,\mgqb <\,} \hat{x}\cdot y {\,\mgqb <\,} \sup (\hat{x}\cdot D) \},%\subset \R^3,
\\ \label{tildeK}
{K}_{D,\eta}^{(\hat{x})}&:=&\{y\in \R^3: \inf(\hat{x}\cdot D)- ct_0 + c\eta {\,\mgqb <\,} \hat{x}\cdot y {\,\mgqb <\,} \sup (\hat{x}\cdot D) -ct_0 +c\eta \},%\subset \R^3,
\en }
whose directions are perpendicular to the observation direction $\hat{x}$.
The region $K_{D}^{(\hat{x})}\subset \R^3$ represents the smallest closed strip containing $D$ and perpendicular to the vector $\hat{x}\in\s^2$, whereas $K_{D, \eta}^{(\hat{x})}\subset \R^3$ is a shift of $K_{D}^{(\hat{x})}$ along the direction $\hat{x}$.
By the definition \eqref{tildeK} and the relation
\begin{equation*}
\inf(-\hat{x}\cdot D) = -\sup (\hat{x}\cdot D),\quad \sup (-\hat{x}\cdot D)= -\inf(\hat{x}\cdot D),
\end{equation*}
it is obvious that
\be\label{K-}
K_{D,\eta}^{(-\hat{x})}:=\{y\in \R^3: \inf(\hat{x}\cdot D)+ ct_0 - c\eta {\,\mgqb <\,} \hat{x}\cdot y {\,\mgqb <\,} \sup (\hat{x}\cdot D) +ct_0 -c\eta \}\subset \R^3.
\en
Moreover, we have
\ben
K_{D}^{(\hat{x})}=
K_{D,\eta}^{(\hat{x})}=K_{D,\eta}^{(-\hat{x})}\quad\mbox{if and only if}\quad \eta=t_0.
\enn

\begin{figure}[H]
\centering
\scalebox{0.9}{
\begin{tikzpicture}
		% grating
		\draw[line width=2.5cm,color=blue!20] (2.5,-3) -- (2.5,3);
		\draw[line width=2.5cm,color=blue!20] (-3,-3) -- (-3,3);
		\draw[line width=2.5cm,color=green!20] (-0.25,-3) -- (-0.25,3);
		\draw (-0.25,2) node [left] {$K_{D}^{(\hat{x})}$};
		\draw (3,2) node [left] {${K}_{D,\eta}^{(\hat{x})}$};
		\draw (-2.5,2) node [left] {${K}_{D,\eta}^{(-\hat{x})}$};
		\draw[->] (-6,0) -- (6,0) node[above] {$x_1$} coordinate(x axis);
		\draw[->] (0,-4) -- (0,4) node[right] {$x_2$} coordinate(y axis);
		\foreach \x/\xtext in {-4,-2, 2, 4}
		\draw[xshift=\x cm] (0pt,1pt) -- (0pt,-1pt) node[below] {$\xtext$};
		\foreach \y/\ytext in {-4,-2, 2, 4}
		\draw[yshift=\y cm] (1pt,0pt) -- (-1pt,0pt) node[left] {$\ytext$};

		\draw[domain = -2:360,very thick][samples = 200] plot({cos(\x)+0.65*cos(2* \x)-0.65},{1.5*sin(\x)});

		\draw [very thick, densely dotted] (-1.5,-3) -- (-1.5,3);
		\draw [very thick, densely dotted] (1,-3) -- (1,3);
		\draw [very thick, densely dotted] (1.25,-3) -- (1.25,3);
		\draw [very thick, densely dotted] (3.75,-3) -- (3.75,3);
		\draw [very thick, densely dotted] (-4.25,-3) -- (-4.25,3);
		\draw [very thick, densely dotted] (-1.75,-3) -- (-1.75,3);
		\draw (-0.5,-0.4) node [below] {$D$};

		\draw [thick,->] (4.1,1) -- (4.5,1);
		\draw (4.5,1) node[right] {$\hat{x} = (1,0)$};

		\end{tikzpicture}
		}
		\caption{Illustration of the strips $K_{D}^{(\hat{x})}$ (green area), ${K}_{D,\eta}^{(-\hat{x})}$ and ${K}_{D,\eta}^{(\hat{x})}$ (blue area) with {$\eta = 2.75$ and} $\hat{x} = (1,0)$ in the $Ox_1x_2$-plane.  The strip $K_{D, \eta}^{(\hat{x})}$ lies on the right hand side of  $K_{D}^{(\hat{x})}$ if $\eta>t_0$ and on the left hand side if   $\eta<t_0$.}
		\label{strip}
		\end{figure}
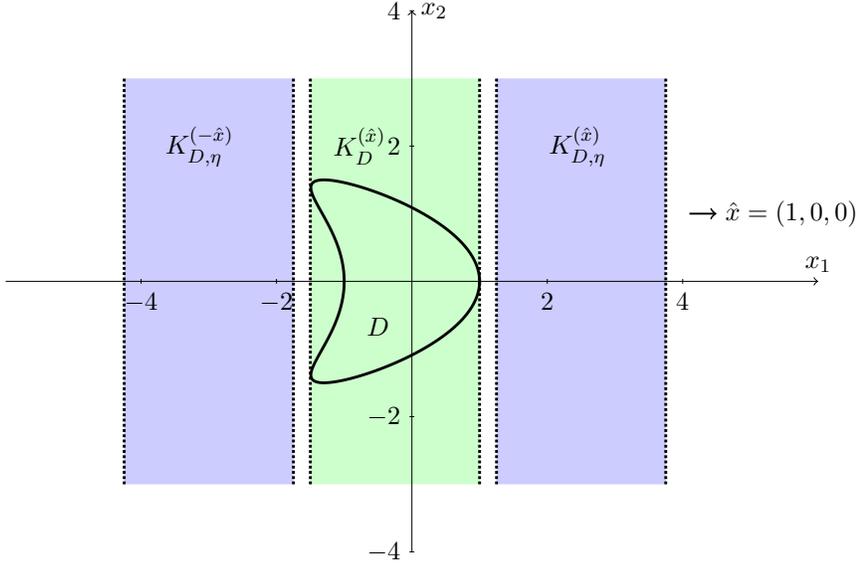

The main result of this section is stated as follows.
\begin{theorem}\label{lem3.4}
For any $\eta \in \R$, we have $y \in {K}_{D,\eta}^{(\hat{x})}$ if and only if $\phi^{(\hat{x})}_{y,\eta} \in  {\range (L)}$.
\end{theorem}
{\mgqb 
\begin{proof} (i)
Assume $y\in {K}_{D,\eta}^{(\hat{x})}$, that is, $\inf(\hat{x}\cdot D) < \hat{x}\cdot y+ct_0-c\eta  < \sup(\hat{x}\cdot D)$. There exists $\epsilon>0$ and $y^\ast\in D$ such that $\hat{x}\cdot y = \hat{x}\cdot y^\ast -ct_0 +c\eta$ and 
$$
\phi^{(\hat{x})}_{y,\eta}(k):=e^{-ik[c^{-1}(\hat{x}\cdot y^* -ct_0+c\eta)-\eta]} = e^{-ik(c^{-1}\hat{x}\cdot y^* -t_0)}=\int_{D}e^{-i k\,  (c^{-1}\hat{x}\cdot z -t_0)}\delta(z-y^*)dz =L\delta(z-y^*).
$$
Since $\delta(z-y^*) \in H^{-2}(D)$, we have $\phi^{(\hat{x})}_{y,\eta} \in \range L.$

(ii) Assume $\phi^{(\hat{x})}_{y,\eta}\in {\range\, (L)}$. Suppose $\phi^{(\hat{x})}_{y,\eta}=Lv, \, v\in H^{-2}(D).$ From Lemma \ref{supplu}, the support of $\mathcal{F}^{-1} (Lv)$ satisfies  $\mbox{supp } \mathcal{F}^{-1} (Lv)\subset \left(c^{-1}\inf(\hat{x}\cdot {D})-t_0,\, c^{-1}\sup(\hat{x}\cdot {D})-t_0\right)$, implying that
$$\mbox{supp }\, (\mathcal{F}^{-1} \phi^{(\hat{x})}_{y,\eta})\subset \left(c^{-1}\inf(\hat{x}\cdot {D})-t_0,\, c^{-1}\sup(\hat{x}\cdot {D})-t_0\right).$$
This together with the fact that $(\mathcal{F}^{-1} \phi^{(\hat{x})}_{y,\eta})(k) = \delta (k-c^{-1}\hat{x}\cdot y +\eta)$ yields
\ben
c^{-1}\inf(\hat{x}\cdot {D})-t_0 <  c^{-1}\hat{x}\cdot y -\eta < c^{-1}\sup(\hat{x}\cdot {D})-t_0,
\enn
that is,
$y\in {K}_{D,\eta}^{(\hat{x})}$.
\end{proof}
}

\section{Indicator functions and uniqueness}\label{sec4}

The aim of this section is to design indicator functions for characterizing the strip ${K}_{D}^{(\hat{x})}$ and the excitation instant $t_0$, which also imply uniqueness results by using multi-frequency data at any two opposite directions.

 Recall from Theorem \ref{lem3.4} that the test function $\phi^{(\hat{x})}_{y,\eta}$ can be utilized to characterize $K_{D,\eta}^{(\hat{x})} $ through \eqref{RI}. Hence, we define the auxiliary indicator function :
\be\label{indicator4}
	I_{\eta}^{(\hat{x})}(y):=\sum_{n=1}^\infty\frac{|\langle \phi^{(\hat{x})}_{y,\eta}, \psi_n^{(\hat{x})} \rangle|_{L^2(0, K )}^2}{ |\lambda_n^{(\hat{x})}|}, \qquad y\in \R^3.
\en
By Theorem \ref{lem3.4} and Picard's theorem, there holds
		%\begin{remark}\label{I-eta}
		$I_{\eta}^{(\hat{x})}(y) < +\infty$ if and only if $y\in K_{{D},\eta}^{(\hat{x})} $ for all $\eta \in {\R}$.
	%	\end{remark}
Then, we define the indicator function determined by a pair of opposite observation directions :
		\begin{equation}\label{indicator1}
		W_{\eta}^{(\hat{x})} (y) = \left[I_{\eta}^{(\hat{x})}(y) +I_{\eta}^{(-\hat{x})}(y) \right]^{-1}.
		\end{equation}
		Combining the range identity \eqref{RI} and Theorem \ref{lem3.4} yields

		{\color{mgq}
		\begin{theorem}\label{Th:factorization}
		\begin{equation*}
		W_{\eta}^{(\hat{x})}(y) = \left\{
		\begin{aligned}
		&\mbox{a finite positive number} &, \mbox{ if } y\in K_{D,\eta}^{(\hat{x})} \cap K_{D,\eta}^{(-\hat{x})},\\
		& 0 &, \mbox{ if } y\notin {K_{D,\eta}^{(\hat{x})} \cap K_{D,\eta}^{(-\hat{x})} }.
		\end{aligned}
		\right.
		\end{equation*}
		\end{theorem}
		}

		\begin{proof}
		(i) For $y\in K_{{\color{mgq}D},\eta}^{(\hat{x})} \cap K_{{\color{mgq}D},\eta}^{(-\hat{x})}$, there holds that $\phi^{(\hat{x})}_{y,\eta}\in \range (L^{(\hat{x})})$ and $\phi^{(-\hat{x})}_{y,\eta}\in \range (L^{(-\hat{x})})$ for all $\eta \in {\mgq \R}$ by Theorem \ref{lem3.4}. From the range identity \eqref{RI} and Picard's theorem (Theorem A. 2), we have that $I_{\eta}^{(\hat{x})}(y)<+\infty$ and $I_{\eta}^{(-\hat{x})}(y)<+\infty$. Thus, $W_{\eta}^{(\hat{x})}(y)$ is a finite positive number.

		(ii) Suppose that $y\notin {K_{{\color{mgq}D},\eta}^{(\hat{x})} \cap K_{{\color{mgq}D},\eta}^{(-\hat{x})}}$. Without loss of generality we  assume that $y\notin {K_{{\color{mgq}D},\eta}^{(\hat{x})} }$. In this case, $\phi^{(\hat{x})}_{y,\eta}\notin \range (L^{(\hat{x})})$ and $I_{\eta}^{(\hat{x})}(y)=+\infty$. Thus, $W_{\eta}^{(\hat{x})}(y)=0$.
		\end{proof}

		 We observe that ${K}_{D,\eta}^{(\hat{x})}$ intersect with  ${K}_{D,\eta}^{(-\hat{x})}$ when $\eta$ lies in a neighborhood of the excitation instant $t_0$. Define
		 \ben
		 \eta_1:=\inf  \left\{\eta: {K}_{D,\eta}^{(\hat{x})} \cap
		 {K}_{D,\eta}^{(-\hat{x})}  \neq \emptyset  \right\},\quad
		  \eta_2:=\sup \left\{\eta: {K}_{D,\eta}^{(\hat{x})} \cap
		 {K}_{D,\eta}^{(-\hat{x})}  \neq \emptyset  \right\},
		 \enn
		and
		\ben
		h^{(\hat x)}(\eta):=\max_{y\in B_R}W_{\eta}^{(\hat{x})}(y),\quad \eta>0.
		\enn
	Then one can recover the unknown pulse moment $t_0$ by plotting the one-dimensional function $\eta \to 	h^{(\hat x)}(\eta)$.

		\begin{theorem}[Determination of $t_0$]\label{Th:max}
		%The function $h^{(\hat x)}_{t_0}(\eta)=\max_{y\in B_R}W_{\eta}^{(\hat{x})}(y)$ {\color{mgq} associated with the pulse moment $t_0$} fulfills that
		We have $t_0 = \frac{\eta_1+\eta_2}{2}$ and
		\ben
		h^{(\hat x)}(\eta)=\left\{\begin{array}{lll}
		\geq 0, &&\mbox{if}\quad\eta\in(\eta_1,\eta_2),\\
		0, && \mbox{if}\quad\eta\notin(\eta_1,\eta_2).
		\end{array}\right.
		\enn
	\end{theorem}

		\begin{proof}
		From the definitions of $K_{{\color{mgq}D},\eta}^{(\hat{x})}$ in \eqref{tildeK} and $K_{{\color{mgq}D},\eta}^{(-\hat{x})}$ \eqref{K-}, one  obtains
		$$\eta_1 = t_0 -\frac{\sup(\hat{x}\cdot {\color{mgq}D})-\inf(\hat{x}\cdot {\color{mgq}D})}{2} \mbox{ and } \eta_2 = t_0 +\frac{\sup(\hat{x}\cdot {\color{mgq}D})-\inf(\hat{x}\cdot {\color{mgq}D})}{2}$$
		after a simple calculation. Hence, $t_0 = (\eta_1+\eta_2)/2$. The indicating behavior of $h^{(\hat{x})}$ follows directly from Theorem \ref{Th:factorization}.
		\end{proof}

		Having determined the pulse moment $t_0$ from Theorem \ref{Th:max}, one can characterize the closed strip $K_{{\color{mgq}D}}^{(\hat{x})}$ through the indicator function $I^{(\hat{x})}(y)$:
		\begin{equation}\label{I-to}
		I^{(\hat{x})}(y) :=\left[\sum_{n=1}^\infty\frac{|\langle \phi^{(\hat{x})}_{y,t_0}, \psi_n^{(\hat{x})} \rangle|_{L^2(0, K )}^2}{ |\lambda_n^{(\hat{x})}|}\right]^{-1}, \qquad y\in \R^3.
		\end{equation}

		\begin{theorem}[Determination of the strip $K_{{\color{mgq}D}}^{(\hat{x})}$]\label{Th:kd}
		The indicator function $I^{(\hat{x})}(y)$ fulfills that
		\ben
		I^{(\hat{x})}(y)=\left\{\begin{array}{lll}
		> 0, &&\mbox{if}\quad y\in K_{{\color{mgq}D}}^{(\hat{x})},\\
		0, && \mbox{if}\quad y\notin K_{{\color{mgq}D}}^{(\hat{x})}.
		\end{array}\right.
		\enn
		\end{theorem}

		\begin{proof} It is a direct consequence of
	Theorem \ref{Th:factorization} when $\eta = t_0$.
		\end{proof}

		According to  Theorems \ref{Th:max} and \ref{Th:kd}, we get uniqueness results for identifying the strip $K_{{\color{mgq}D}}^{(\hat{x})}$ and the pulse moment $t_0$, which are summarized as follows.
		\begin{theorem}[Uniqueness at two opposite directions]
		The strip $K_{{\color{mgq}D}}^{(\hat{x})}$ and the pulse moment $t_0$ can be uniquely determined by the multi-frequency far-field data $\{u^{\infty}(\pm \hat{x},\omega):\omega\in [\omega_{\min},\omega_{\max}]\}$.
		\end{theorem}

		In the case of sparse observation directions $\{\hat{x}_m: m=1,2,\cdots, M\}$, we shall make use of the following indicator function:
		\begin{equation}\label{W2}
		I(y){= \left[\sum_{m=1}^M \frac{1}{I^{(\hat{x}_m)}(y) }\right]^{-1}}= \left[\sum_{m=1}^M\sum_{n=1}^N\frac{|\langle \phi^{(\hat{x}_m)}_{y,t_0}, \psi_n^{(\hat{x}_m)} \rangle|_{L^2(0, K )}^2}{ |\lambda_n^{(\hat{x}_m)}|}  \right]^{-1}, y\in \R^3,
		\end{equation}
		{\color{mgq} where $\left(\lambda_n^{(\hat{x}_m)}, \psi_n^{(\hat{x}_m)}\right)$ are the eigensystems of the operators $F^{(\hat{x}_m)}_{\#}$.}
Define the $\Theta$-convex hull of ${\color{mgq}D}$ associated with the directions $\{\hat{x}_m: m=1,2,\cdots, M\}$ as (see Figure \ref{inter} for $M=2$)
		\ben
		\Theta_{{\color{mgq}D}}:=\bigcap_{m=1,2,\cdots, M} K_{{\color{mgq}D}}^{(\hat{x}_m)}.
		\enn

\begin{figure}[H]
	\centering
	\scalebox{0.8}{
		\begin{tikzpicture}
			\draw[line width=3cm,color=gray!20] (-4,0) -- (4,0);
		%	\draw[line width=2.5cm,color=blue!20] (-3,-3) -- (-3,3);
			
			\draw[line width=2.5cm,color=gray!20] (-0.25,-3.5) -- (-0.25,3.5);
			
			\draw (0.5,3) node [left] {$K_{D}^{(\hat{x}_1)}$};
			
				\draw (3.5,0) node [left] {$K_{D}^{(\hat{x}_2)}$};
				
			%\draw (3,2) node [left] {${K}_{D,\eta}^{(\hat{x})}$};
		%	\draw (-2.5,2) node [left] {${K}_{D,\eta}^{(-\hat{x})}$};
		%	\draw[->] (-6,0) -- (6,0) node[above] {$x_1$} coordinate(x axis);
		%	\draw[->] (0,-4) -- (0,4) node[right] {$x_2$} coordinate(y axis);
		%	\foreach \x/\xtext in {-4,-2, 2, 4}
		%	\draw[xshift=\x cm] (0pt,1pt) -- (0pt,-1pt) node[below] {$\xtext$};
		%	\foreach \y/\ytext in {-4,-2, 2, 4}
		%	\draw[yshift=\y cm] (1pt,0pt) -- (-1pt,0pt) node[left] {$\ytext$};
			
			\draw[domain = -2:360,very thick][samples = 200] plot({cos(\x)+0.65*cos(2* \x)-0.65},{1.5*sin(\x)});
			
			\draw [very thick, densely dotted] (-1.5,-3.5) -- (-1.5,3.5);
			\draw [very thick, densely dotted] (1,-3.5) -- (1,3.5);
			\draw [very thick, densely dotted] (-4,1.5) -- (4,1.5);
			\draw [very thick, densely dotted] (-4,-1.5) -- (4,-1.5);
			%\draw [very thick, densely dotted] (-4.25,-3) -- (-4.25,3);
		%	\draw [very thick, densely dotted] (-1.75,-3) -- (-1.75,3);
			\draw (-0.5,0) node [below] {$D$};
			
		%	\draw [thick,->] (4.1,1) -- (4.5,1);
		%	\draw (4.5,1) node[right] {$\hat{x} = (1,0)$};
			
		\end{tikzpicture}
	}
	\caption{Intersection of the strips $K_{D}^{(\hat{x}_1)}$  and
	$K_{D}^{(\hat{x}_2)}$ with $\hat{x}_1=(1,0)$ and $\hat{x }_2=(0,1)$.}
	\label{inter}
\end{figure}

		\begin{theorem}\label{TH:hull}
		We have $I(y) > 0$ if $y\in \Theta_{{\color{mgq}D}}$ and $I(y)=0$ if
		$y\notin \Theta_{{\color{mgq}D}}$.
		\end{theorem}
		\begin{proof}
		If $y\in \Theta_{{\color{mgq}D}}$, then $y\in K_{{\color{mgq}D}}^{(\hat x_m)}$ for all $m=1,2,...,M$, yielding that $\hat x_m \cdot y \in \hat x_m \cdot {{\color{mgq}D}}$. Hence, one deduces from  Theorem \ref{Th:kd} that  $0< I^{(\hat x_m)}(y)<\infty$ for all $m=1,2,...,M$, implying that $I(y) > 0$. On the other hand, if $y\notin \Theta_{{\color{mgq}D}}$, there must exist some unit vector $\hat{x}_l$ such that $y\notin K_{{\color{mgq}D}}^{(\hat x_l)}$. Again, using Theorem \ref{Th:kd} we get $$
		[I^{(\hat{x}_l)}(y)]^{-1}=\infty,$$ which proves $I(y)=0$ for $y\notin \Theta_{{\color{mgq}D}}$.
		\end{proof}
The values of $I(y)$ are expected to be large for $y\in {\Theta_{{\color{mgq}D}}}$ and small for those $y\notin \Theta_{{\color{mgq}D}}$. % because $\Theta_D\subset K_D^{(\hat{x}_j)}$ for all $j=1,2,\cdots, M$.
As a by-product of the above factorization method, we obtain an uniqueness result with multi-frequency far-field data. Denote by $\mbox{ch}(D)$ the convex hull of $D$, that is,  the intersections of all half spaces containing $D$.

\begin{theorem}
Under the Assumption \ref{sump}, both $\mbox{ch}({{\color{mgq}D}})$ and $t_0$ can be uniquely determined by the multi-frequency far-field patterns $\{u^\infty(\hat{x}, \omega):  \omega \in[\omega_{\min}, \omega_{\max}],\, \hat{x}\in \s^2\}.$
\end{theorem}

\begin{corollary}
Let the Assumption \ref{sump} hold true. The excitation instants $t_j$ and the $\Theta$-convex hull of supports $D_j$ of $S(\cdot-a(t_j))$ ($j=0,1,2,\cdots J$) can be determined by the multi-frenquency far-field data $\{u_j^\infty(\hat{x}_m, \omega): \omega\in[\omega_{\min}, \omega_{\max}], \, m=1,2,\cdots, M, j=1,2,\cdots,J \}$.
\end{corollary}
\begin{proof}
The time signals excited at different moments can be separated under the Assumption \ref{sump}. The wave fields generated at the
 radiating moment $t_j$ can be reduced to the wave equation
\begin{equation*}
\left\{
\begin{aligned}
&c^{-2}\frac{\partial^2 U}{\partial t^2} = \Delta U + S(x-a(t))\delta(t-t_j), \quad &&(x, t) \in \mathbb{R}^3 \times \mathbb{R}_+,\\
&U(x,0)=\partial_t U(x,0) = 0, &&x\in \mathbb{R}^3.
\end{aligned}
\right.
\end{equation*}
Then we can use the far-field data $u_j^\infty(\hat{x}_m, \omega)$ (that is, the inverse Fourier transform of $U^\infty(\hat{x}_m, t))$ to determine the pulse moment $t_j$ and the $\Theta$-convex hull of $D_j$,  by repeating the arguments for proving Theorems \ref{Th:max}, \ref{Th:kd} and \ref{TH:hull}.
\end{proof}

\section{Numerical examples}\label{num}
In this section, a couple of  numerical examples in $\R^2$ and $\R^3$
will be presented to verify the effectiveness of the algorithm proposed in this paper. Our algorithm also holds in two dimensions. All numerical examples are implemented by MATLAB. We aim to determine the pulse moment and extract the information of the shape and locations of a moving extended time-dependent source through the multi-frequency far-field measurement from a single or finite observation direction(s). In the practical scenarios, the multi-frequency data for reconstruction can be obtained via  the inverse Fourier transform of the the time-domain data.

We briefly summarize the inversion procedure. Discretize the  frequency interval $[\omega_{\min},\omega_{\max}]$ as
$$\omega_n=(n-0.5)\Delta \omega, \quad \Delta \omega:=\frac{K}{N}, \quad n=1,2,\cdots,N,$$ where $K$ is the  half of the bandwidth.
The far-field operator in (\ref{def:F}) can be approximated by applying the numerical rectangular rule:
\begin{equation}
({F}^{(\hat x)}\phi)(\tau_n) \approx \sum_{m=1}^{N} u^{\infty}(\hat x, \kappa+\tau_n-s_m)\phi(s_m)\Delta \omega,
\end{equation}
where $\tau_n:=n\Delta \omega$ and $s_m:=(m-0.5)\Delta \omega$, $n,m=1,2,\cdots,N$.
A discrete approximation of the far-field operator ${F}^{(\hat x)}$ is given  by the Toeplitz matrix
\be \label{matF}
F^{(\hat x)}:= \Delta \omega \begin{pmatrix}
u^{\infty}({\hat x},\kappa+\omega_1) & u^{\infty}({\hat x},\kappa-\omega_1) & \cdots   & u^{\infty}({\hat x},\kappa-\omega_{N-1})  \\
u^{\infty}({\hat x},\kappa+\omega_2) & u^{\infty}({\hat x},\kappa+\omega_1) & \cdots  &u^{\infty}({\hat x},\kappa-\omega_{N-2})   \\
\vdots & \vdots  &\vdots  &\vdots \\
u^{\infty}({\hat x},\kappa+\omega_{N-1}) & u^{\infty}({\hat x},\kappa+\omega_{N-2}) &  \cdots  & u^{\infty}({\hat x},\kappa-\omega_1)\\
u^{\infty}({\hat x},\kappa+\omega_N) & u^{\infty}({\hat x},\kappa+\omega_{N-1}) &  \cdots & u^{\infty}({\hat x},\kappa+\omega_1)\\
\end{pmatrix},
\en
where $F^{(\hat x)}$ is a $N\times N$ complex matrix. Here, we use {$2N-1$ samples $u^{\infty}(\hat x, \kappa+\omega_n), n=1,2,\cdots,N$ and $u^{\infty}(\hat x, \kappa-\omega_n), n=1,2,\cdots,N-1$}, of the far-field pattern.
\noindent Denoting by  $\left\{ ( {\tilde \lambda^{( x)}_n}, \psi^{(x)}_n): n=1,2,\cdots,N \right\}$ an eigen-system of the matrix $F^{(\hat x)}$ (\ref{matF}), then one deduces that  an eigen-system of the matrix $(F^{(\hat x)})_\#:= |\real (F^{(\hat x)})|+|\ima(F^{(\hat x)})|$ is $\left\{ ( \lambda^{(\hat x)}_n, \psi^{(\hat x)}_n): n=1,2,\cdots,N \right\}$, where $ \lambda^{(\hat x)}_n:=|\real (\tilde \lambda^{(\hat x)}_n)| +|\ima (\tilde \lambda^{( \hat x)}_n)|$.   In another way,  we may also use the \textcolor{ghx}{SPD} discrete approximation $F^{(\hat x)}_\#$.
For any test point $y\in \R^3$, the test function $\phi_{y,\eta}^{(\hat x)}$  (\ref{Test}) can be discretized as
\be \label{testn}
\phi_{y,\eta}^{(\hat x)}:= \left( e^{-i\tau_1(c^{-1}\hat x\cdot y-\eta)}, e^{-i\tau_2(c^{-1}\hat x\cdot y-\eta)},..., e^{-i\tau_N (c^{-1}\hat x \cdot y-\eta)}  \right).
\en
We approximate the auxiliary indicator function $I_{\eta}^{(\hat{x})}(y)$   (\ref{indicator4}) for reconstructing the strip $K_{D,\eta}^{(\hat{x})}$  by
\be\label{indicator5}
I_{\eta}^{(\hat{x})}(y)\sim \sum_{n=1}^N\frac{| \phi^{(\hat{x})}_{y,\eta} \cdot \psi_n^{(\hat{x})}|^2}{ |\lambda_n^{(\hat{x})}|}, \qquad y\in \R^3,
\en
where $\cdot$ denotes the inner product in $\R^3$ and $N$ is consistent with the dimension of the Toeplitz matrix \eqref{matF}.
Furthermore, an approximation of the indicator function $W^{(\hat x)}_\eta(y)$ in (\ref{indicator1}) for reconstructing the strip $K_{D,\eta}^{(\hat{x})} \cap K_{D,\eta}^{(-\hat{x})}$  is
\begin{equation}\label{indicator6}
W_{\eta}^{(\hat{x})}(y):=\left[I_{\eta}^{(\hat{x})}(y) +I_{\eta}^{(-\hat{x})}(y) \right]^{-1}\sim
\left[\sum_{n=1}^N\frac{| \phi^{(\hat{x})}_{y,\eta} \cdot \psi_n^{(\hat{x})} |^2}{ |\lambda_n^{(\hat{x})}|} + \frac{|\phi^{(-\hat{x})}_{y,\eta} \cdot \psi_n^{(-\hat{x}} )|^2}{ |\lambda_n^{(-\hat{x})}|} \right]^{-1}.
\end{equation}

%\begin{equation} \label{indicator6}
%W^{(\hat{x})}(y):=\left[I_{\eta}^{(\hat{x})}(y) +I_{\eta}^{(-\hat{x})}(y) \right]^{-1}\sim
%\left[\sum_{n=1}^N\frac{| \phi^{(\hat{x})}_{y,\eta} \cdot \psi_n^{(\hat{x})}|^2}{ |\lambda_n^{(\hat{x})}|} + \sum_{n=1}^N\frac{| \phi^{(-\hat{x})}_{y,\eta} \cdot \psi_n^{(-\hat{x})}|^2}{ |\lambda_n^{(-\hat{x})}|} \right]^{-1}.
%\end{equation}

\noindent The indicator function  for reconstructing the support of the source  with multiple observation directions can be approximated as
\begin{equation} \label{indicator7}
I(y){= \left[\sum_{m=1}^M \frac{1}{I^{(\hat{x}_m)}(y) }\right]^{-1}}\sim \left[\sum_{m=1}^M\sum_{n=1}^N\frac{| \phi^{(\hat{x}_m)}_{y,t_0} \cdot \psi_n^{(\hat{x}_m)} |^2}{ |\lambda_n^{(\hat{x}_m)}|} + \frac{| \phi^{(-\hat{x}_m)}_{y,t_0} \cdot \psi_n^{(-\hat{x}_m)} |^2}{ |\lambda_n^{(-\hat{x}_m)}|} \right]^{-1}, \; y\in \R^3.
\end{equation}

\noindent Suppose $k_{\min}=0$ for the sake of simplicity, then
the bandwidth of frequencies can be extended from $(0,k_{\max})$ to $(-k_{\max}, k_{\max})$ by $u_j^{\infty}(\hat x, -k)=\overline{u_j^{\infty}(\hat x, k)}$. Thus, one deduces from these new measurement data with $k_{\min}=-k_{\max}$ that $\kappa=0$ and $K=k_{\max}$.  The frequency band is represented by the interval $(0, 3\pi)$ and discretized by  $N=48$ and $\Delta k=\pi/16$. We also assume the source function  $S(x-a(t))=2 (x_1-a_1(t))+3 (x_2-a_2(t))^2+ (x_1-a_1(t)) (x_2-a_2(t))^2+1$ which  has supports for $j=1,...,J$, unless other stated.  In subsections 4.1 and 4.2, we just take pulse moment $t_0$ as an example.

\subsection{Reconstruction of the pulse moment $t_0$}

Firstly, we consider reconstructions of the strip $K_{D,\eta}^{(\hat x)}$ (as defined in $(\ref{tildeK})$) perpendicular to the observation direction by plotting the indicator function $1/I_\eta^{(\hat x)}(y)$. The strip
$K_{D,\eta}^{(\hat x)}$ can be regarded as the result of translating $K_{D}^{(\hat x)}$  along the observational direction $\hat x$ by $\eta-t_0$ units. $K_{D}^{(\hat x)}$ is the smallest strip perpendicular to the observation  direction and containing the source support. It is supposed that the source is supported in $|x|\leq1$ and excited at the moment $t_0$. The testing domain is chosen as $[-6,6]\times[-6,6]$. In Figure \ref{fig:3-1}, we present reconstructions of $K_{D,\eta}^{(\hat x)}$ with various parameter $\eta=1,2,...,6$ from the measurement data of one observation direction $\hat x=(1,0)$ and the source is excited at the moment $t_0=4$. The shape of source support $D$ is highlighted by the pink solid line.
According to the definition of  $K_{D,\eta}^{(\hat x)}$ and $K_{D}^{(\hat x)}$, the two strips are completely identical for $t_0=\eta$.  We obtain  the smallest strip containing the source support and perpendicular to the observation direction for $\eta=t_0=4$ by plotting  $1/I_\eta^{(\hat x)}(y)$ in Figure \ref{fig:3-1}(d).  Figures \ref{fig:3-1}(a)-(c)and (e)-(f) demonstrate that $K_{D,\eta}^{(\hat x)}$  has been shifted by $\eta-t_0=-3,-2,-1,1,2$ units, respectively, along the observation direction compared to $K_{D}^{(\hat x)}$.  All of these confirm the correctness of the theory.

\begin{figure}[H]
\centering
\subfigure[$\eta=1$]{
\includegraphics[scale=0.22]{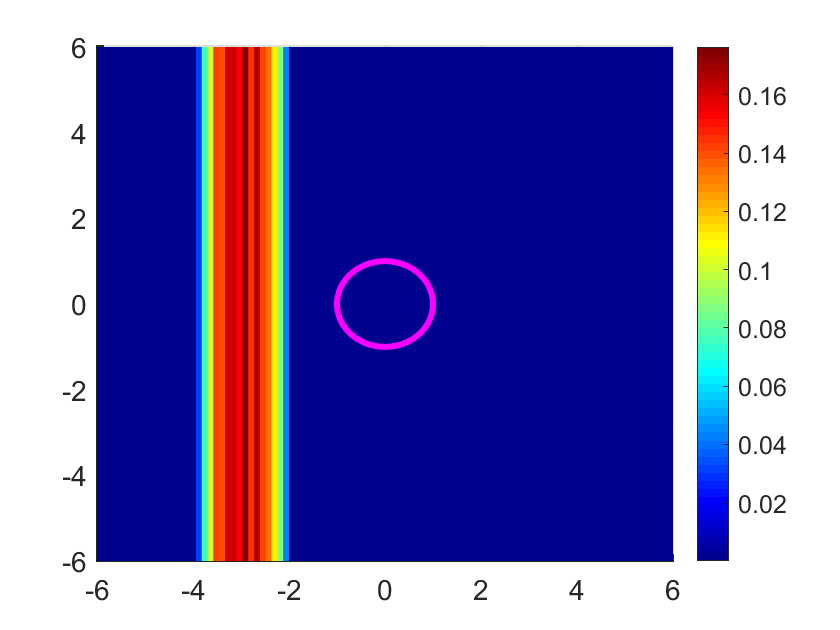}
}
\subfigure[$\eta=2$ ]{
\includegraphics[scale=0.22]{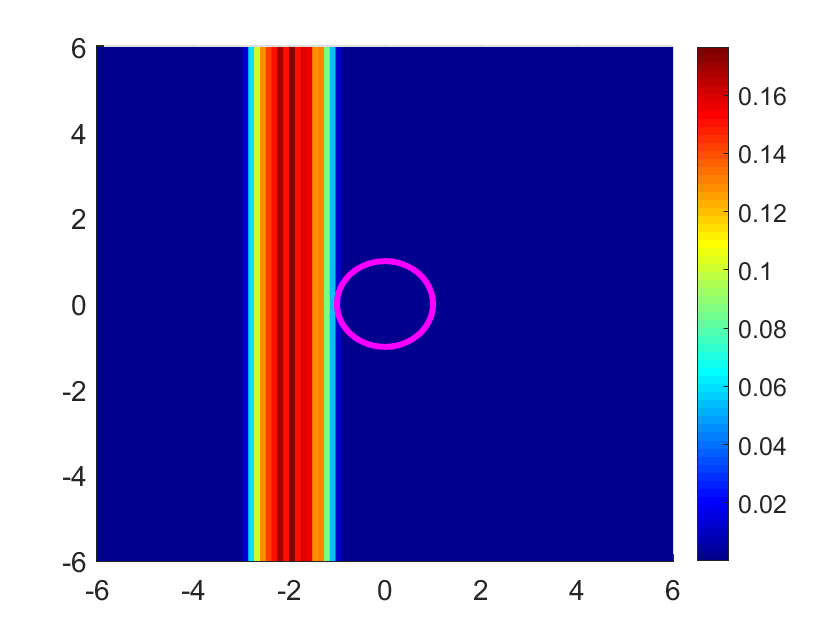}

}
\subfigure[$\eta=3$]{
\includegraphics[scale=0.22]{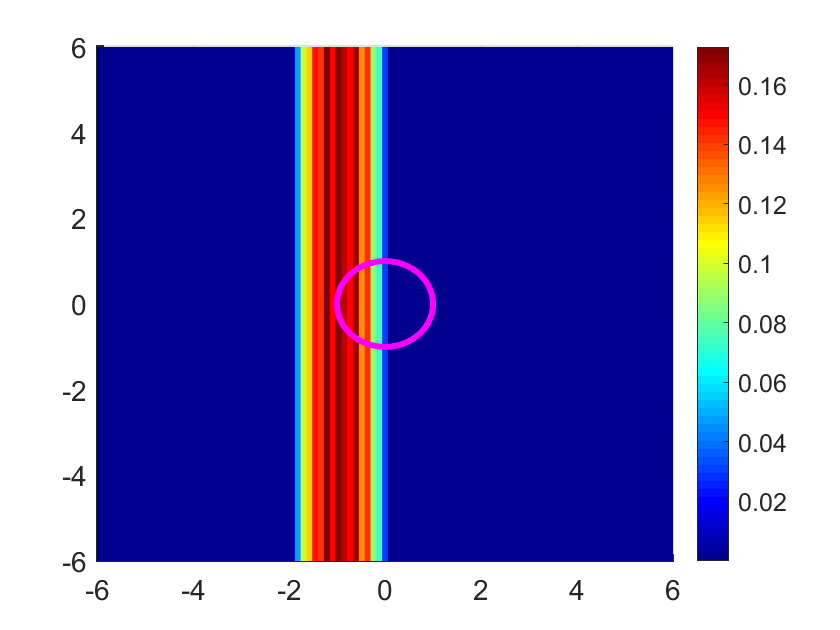}

}
\subfigure[$\eta=4$]{
\includegraphics[scale=0.22]{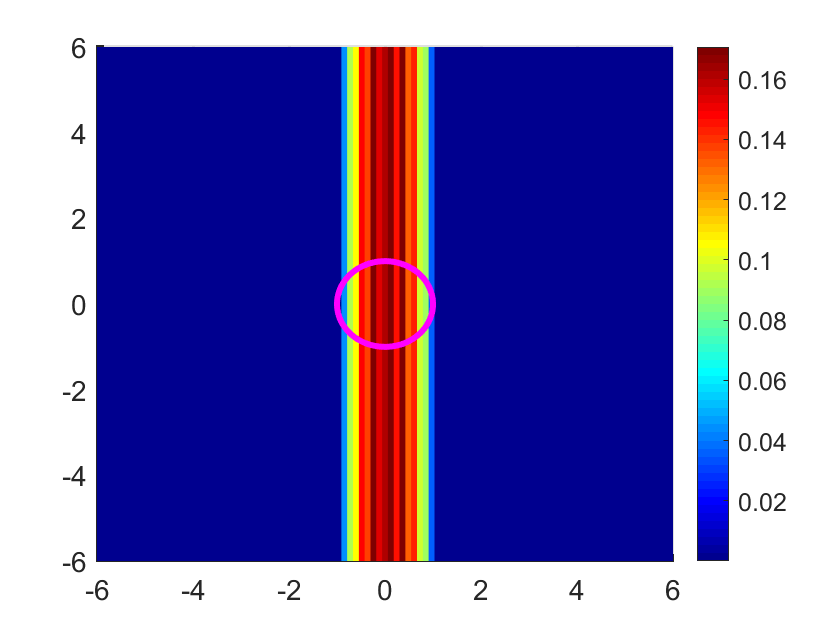}
}
\subfigure[$\eta=5$ ]{
\includegraphics[scale=0.22]{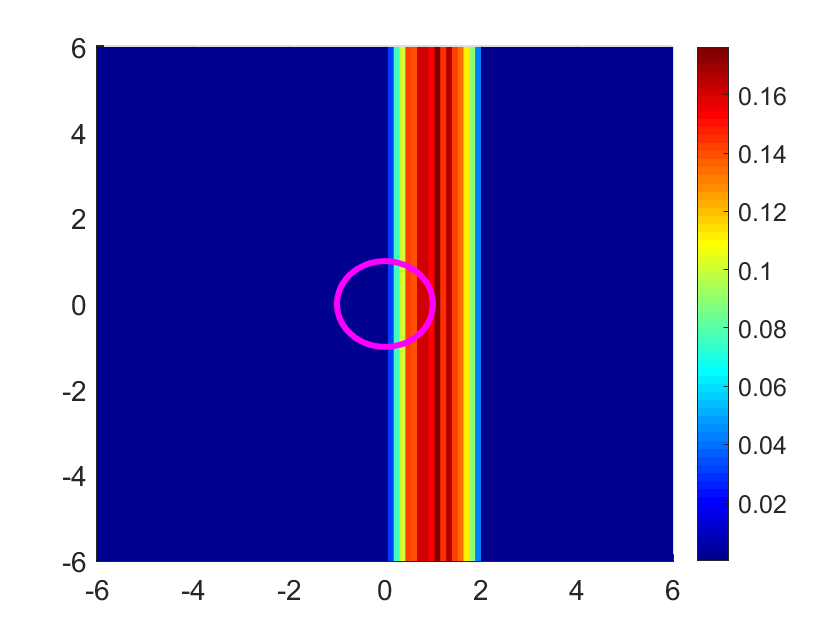}

}
\subfigure[$\eta=6$]{
\includegraphics[scale=0.22]{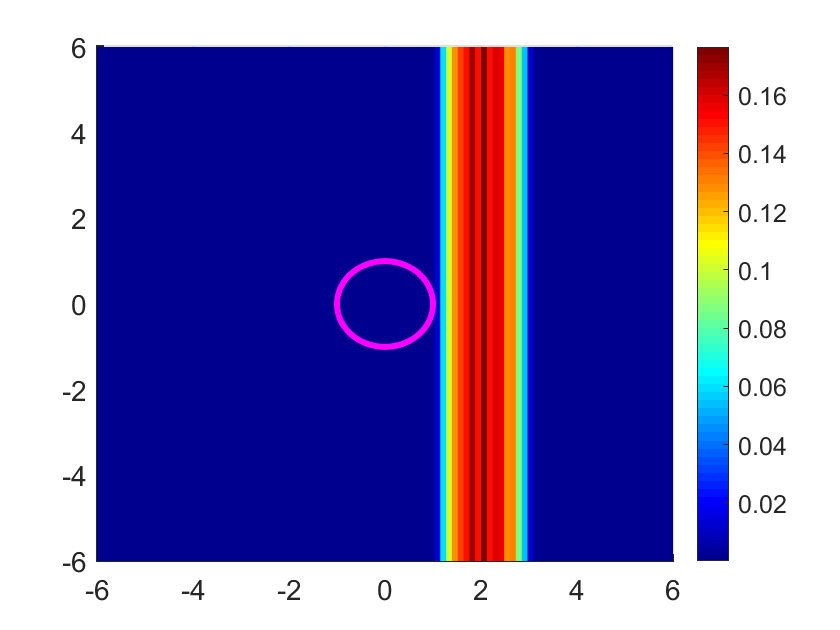}

}
\caption{Reconstruction for $K_{D,\eta}^{(-\hat{x})}$ using multi-frequency far-field data from a single observation direction $\hat x=(1,0)$ with  the auxiliary indicator function $1/I_\eta^{(\hat x)}(y)$. Various $\eta$ are tested and the pulse moment is set at $t_0=4$.
} \label{fig:3-1}
\end{figure}

\noindent We proceed with the source term previously discussed in Figure \ref{fig:3-1}. However, the main object is to reconstruct $K_{D,\eta}^{(\hat{x})} \cap K_{D,\eta}^{(-\hat{x})}$ using multi-frequency far-field data from a pair of observation directions $\hat x=(\pm1,0)$ with  the  indicator function $W_\eta^{(\hat x)}(y)$ defined in \eqref{indicator6} in Figure \ref{fig:3-2}. We test various parameter $\eta$. When  we take $\eta=t_0=4$ in Figure \ref{fig:3-2}(d), the reconstruction illustrate $K_{D,\eta}^{(\hat{x})} \cap K_{D,\eta}^{(-\hat{x})}=K_{D}^{(\hat{x})}$ as proven by our theory. As $\eta=2$, we have  $K_{D,\eta}^{(\hat{x})} \cap K_{D,\eta}^{(-\hat{x})}=\emptyset$. Numerical results in Figure \ref{fig:3-2}(a) indicates that the corresponding indicator
values are consistently much smaller than $10^{-9}$,  which implies that it is not possible to reconstruct partial or whole information  of the  source support. The width of $K_{D,\eta}^{(\hat{x})} \cap K_{D,\eta}^{(-\hat{x})}$ becomes a line at $x=0$ in theory for $\eta=3$ or $5$. The reconstruction results shown in Figures \ref{fig:3-2}(b) and (f) respectively display a straight line at $x=0$. Similarly, the width of $K_{D,\eta}^{(\hat{x})} \cap K_{D,\eta}^{(-\hat{x})}$ equals $1$ theoretically as  $\eta=3.5$ or $5.5$. The numerical results in  Figures \ref{fig:3-2}(c) and (e) further corroborate the accuracy of the theoretical findings.

\begin{figure}[H]
\centering
\subfigure[$\eta=2$]{
\includegraphics[scale=0.22]{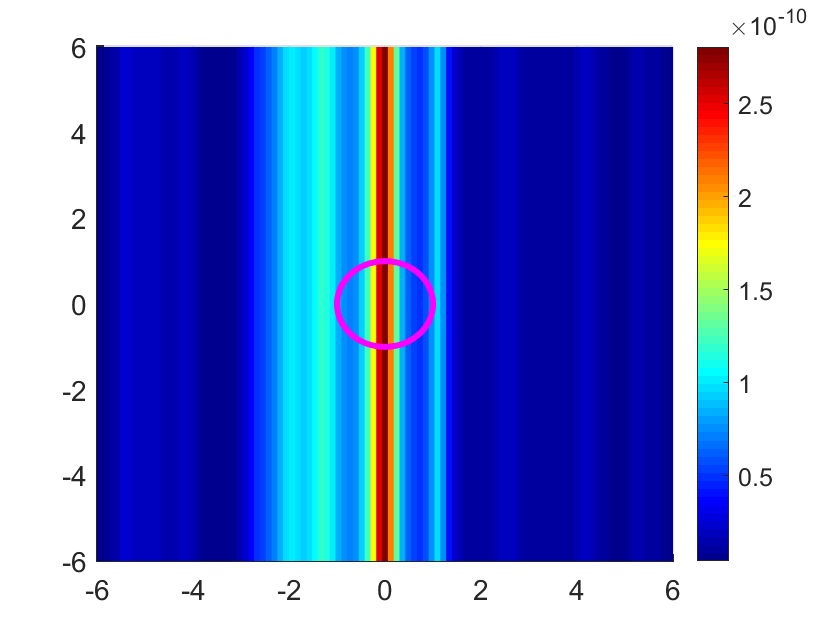}
}
\subfigure[$\eta=3$ ]{
\includegraphics[scale=0.22]{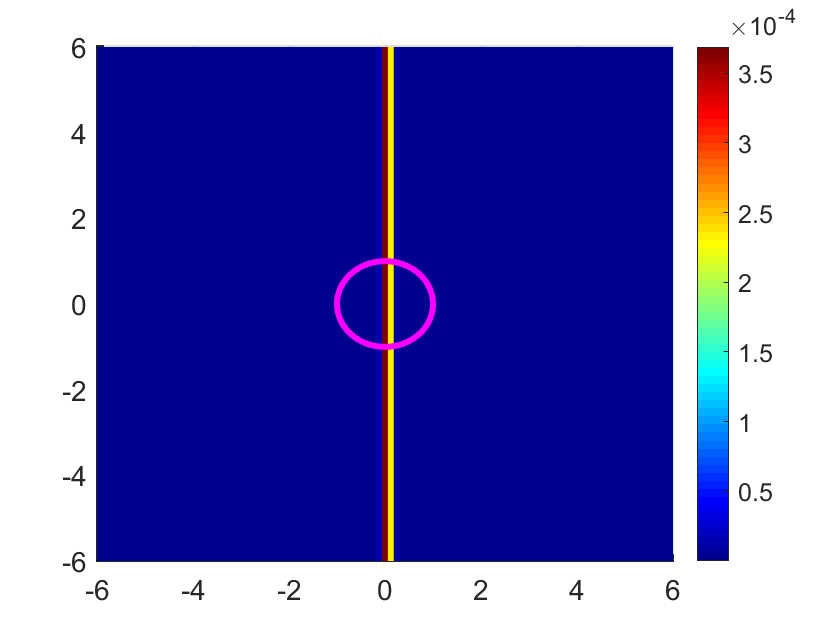}

}
\subfigure[$\eta=3.5$]{
\includegraphics[scale=0.22]{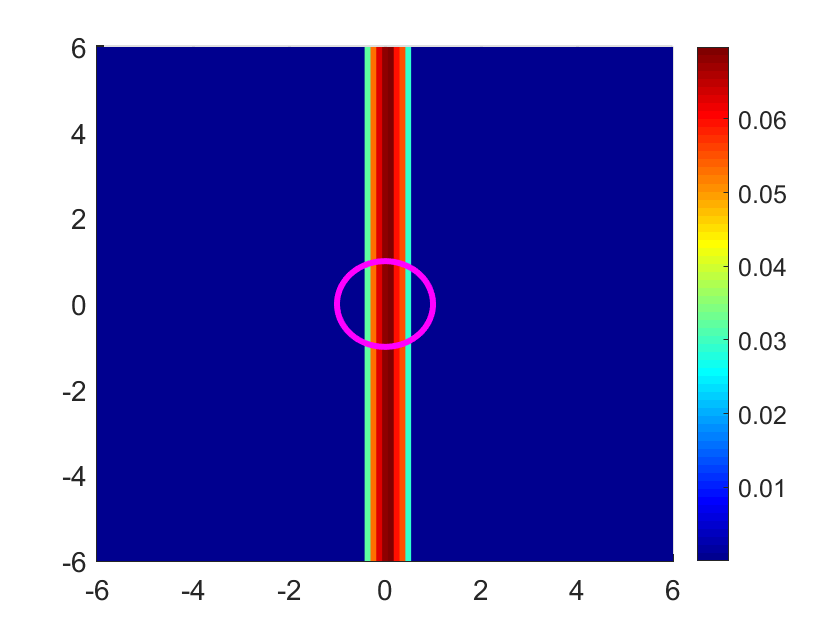}

}
\subfigure[$\eta=4$]{
\includegraphics[scale=0.22]{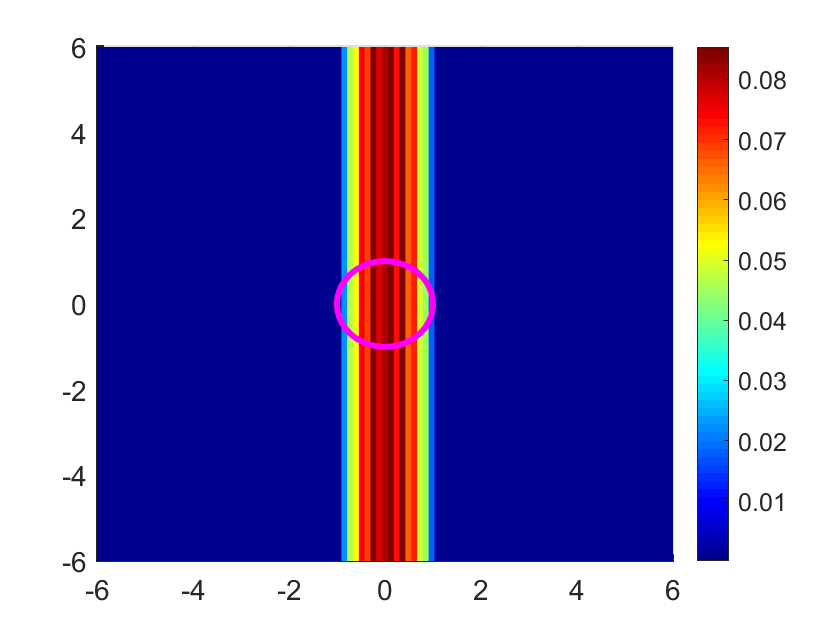}
}
\subfigure[$\eta=4.5$ ]{
\includegraphics[scale=0.22]{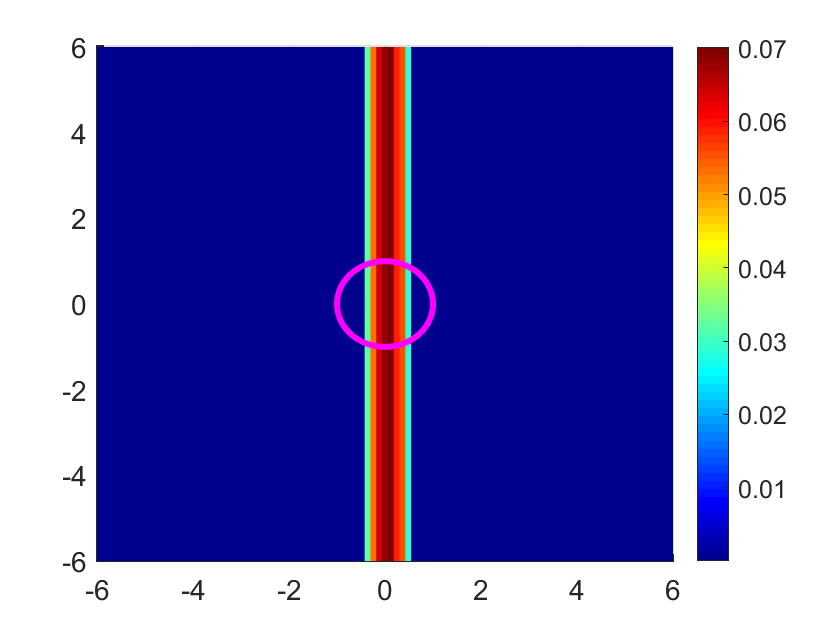}

}
\subfigure[$\eta=5$]{
\includegraphics[scale=0.22]{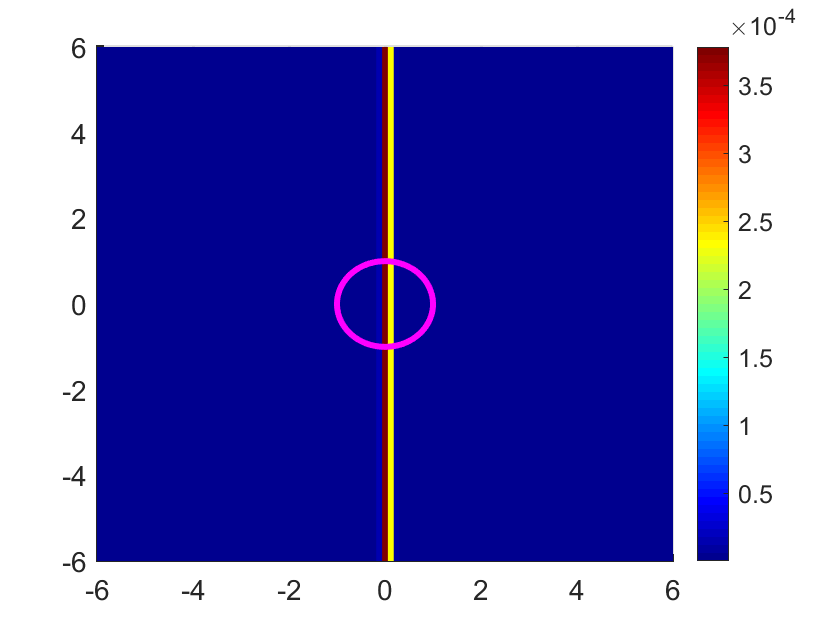}

}
\caption{Reconstruction for $K_{D,\eta}^{(\hat{x})} \cap K_{D,\eta}^{(-\hat{x})}$  using multi-frequency far-field data from a pair of observation directions $\hat x=(\pm1,0)$ with  the  indicator function $W_\eta^{(\hat x)}(y)$. Various $\eta$ are tested and the pulse moment is $t_0=4$.
} \label{fig:3-2}
\end{figure}

\noindent Next, we aim to determine the impulse moment $t_0$ using multi-frequency far-field data from a pair of observation directions $\hat x=(\pm1,0)$ by plotting the one dimensional function $\eta\rightarrow h^{(\hat x)}(\eta)=\max_{y\in B_R}W_{\eta}^{(\hat{x})}(y)$. The source in Figure \ref{fig:3-2} is still considered. The pulse moment $t_0$ is set to be $2,4,6$, respectively, in Figure \ref{fig:3-3}. According to Theorem \ref{Th:max}, we know that $h^{(\hat x)}(\eta)\geq0$ for $\eta\in[\eta_1,\eta_2]$ and $h^{(\hat x)}(\eta)=0$ for $\eta\notin[\eta_1,\eta_2]$, then the impulse moment is $t_0=\frac{\eta_1+\eta_2}{2}$.  Figure \ref{fig:3-3}(a)  demonstrates  that $h^{(\hat x)}(\eta)\geq0$ for $\eta\in[1,3]$ and $^{(\hat x)}(\eta)=0$ for $\eta\notin[1,3]$. Consequently, the numerical impulse moment is $t_0^{num}=2$, which is consistent with the actual impulse moment $t_0=2$. Similarly, Figure \ref{fig:3-3}(b) illustrates that  $h^{(\hat x)}(\eta)\geq0$ for $\eta\in[3,5]$ and $h^{(\hat x)}(\eta)=0$ for $\eta\notin[3,5]$. Therefore, the numerical impulse moment is $t_0^{num}=4$, aligning with the actual impulse moment $t_0=4$.  Lastly, Figure \ref{fig:3-3}(c) shows  that $h^{(\hat x)}(\eta)\geq0$ for $\eta\in[5,7]$ and $h^{(\hat x)}(\eta)=0$ for $\eta\notin[5,7]$. Thus, the numerical impulse moment  is $t_0^{num}=4$, which matches the actual impulse moment $t_0=6$. The aforementioned numerical results demonstrate the validity of the theoretical findings.  The lack of precision near $\eta_1$ and $\eta_2$ is due to the insufficient granularity in the discretization of $\eta$ during numerical computations. Therefore, the numerical results can be improved by increasing the resolution of the \(\eta\) discretization.

\begin{figure}[H]
\centering
\subfigure[$t_0=2$]{
\includegraphics[scale=0.22]{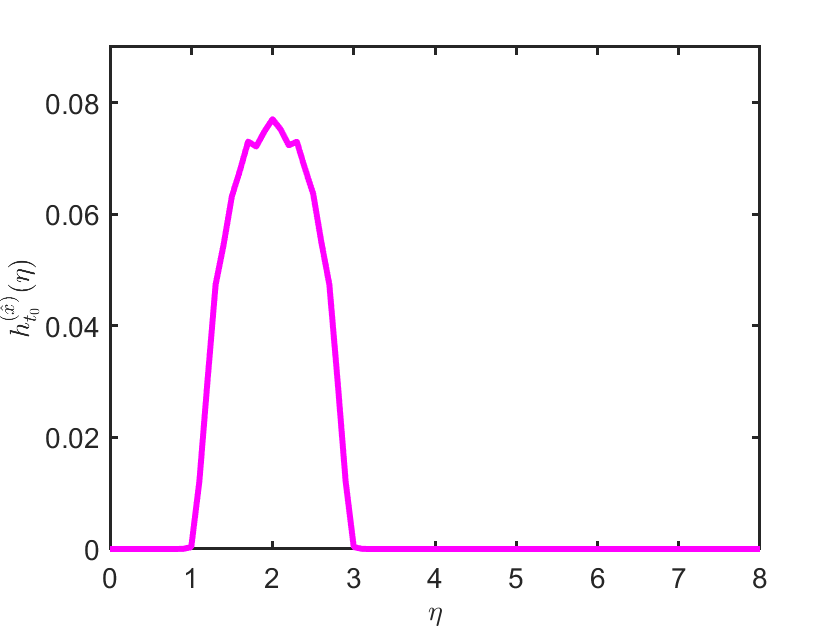}
}
\subfigure[$t_0=4$ ]{
\includegraphics[scale=0.22]{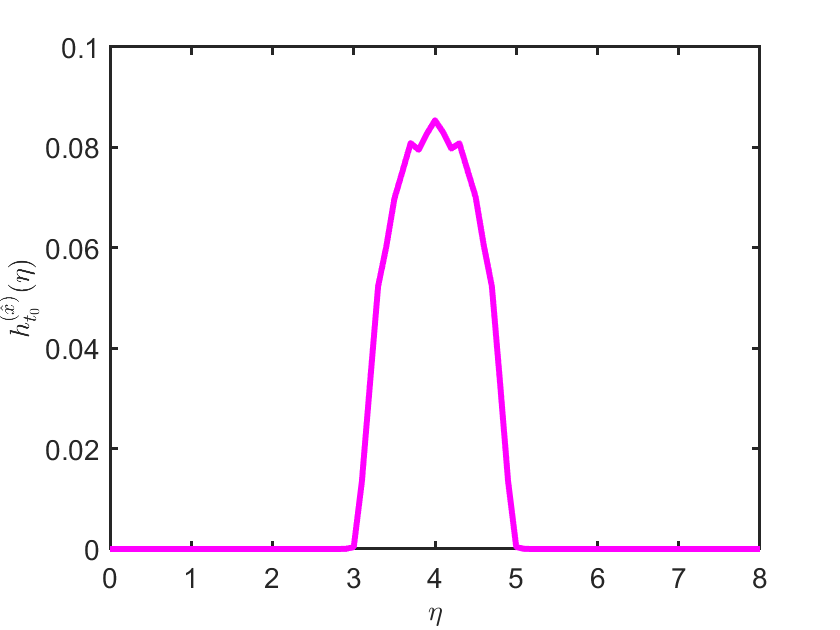}

}
\subfigure[$t_0=6$]{
\includegraphics[scale=0.22]{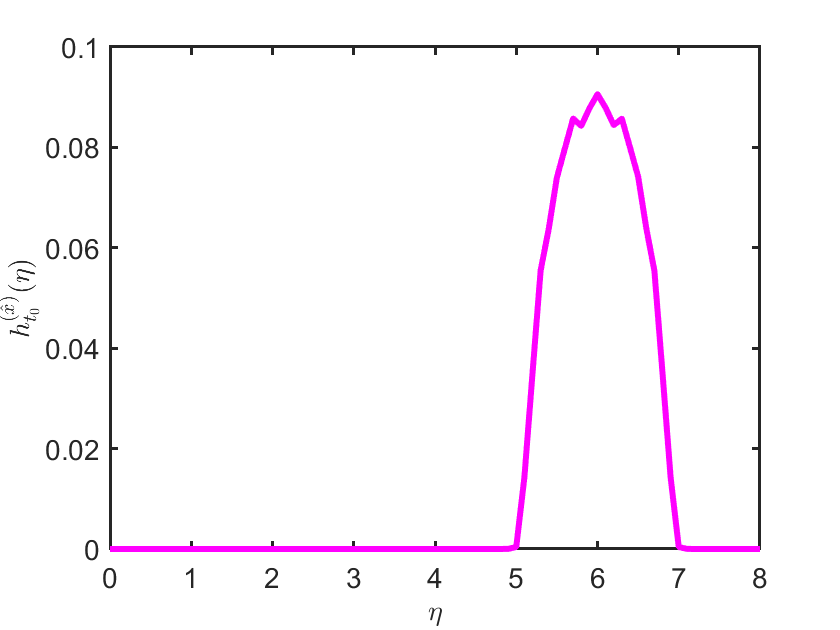}

}
\caption{Determination of the pulse moment $t_0$  using multi-frequency far-field data from a pair of observation directions $\hat x=(\pm1,0)$ with  the  function $h^{(\hat x)}(\eta)$.  The pulse moment $t_0$ is set to be $2,4,6$, respectively.
} \label{fig:3-3}
\end{figure}

To further validate the effectiveness of our proposed algorithm, we tested its robustness in reconstructing  pulse moments across various observation directions.  We reconstruct  $K_{D,\eta}^{(\hat{x})} \cap K_{D,\eta}^{(-\hat{x})}$  using multi-frequency far-field data from a pair of observation directions $\hat x=(\pm\frac{\sqrt{2}}{2}, \pm\frac{\sqrt{2}}{2})$ with  the  indicator function $W_\eta^{(\hat x)}(y)$. Various $\eta$ is tested and the pulse moment is $t_0=4$.
As illustrated in Figure \ref{fig:3-2}, in Figure \ref{fig:3-4} we employed different observation directions and obtained numerical results consistent with those in Figure \ref{fig:3-2}. Specifically, when $\eta=t_0=4$, we achieved $K_{D,\eta}^{(\hat{x})} \cap K_{D,\eta}^{(-\hat{x})}=K_{D}^{(\hat{x})}$ . For $3\leq\eta\leq5$, we obtained an even narrower strip. In other cases, the numerical results are nearly zero in the test domain.
We aim to determine the pulse moment $t_0$  using multi-frequency far-field data from a pair of observation directions $\hat x=(\pm\frac{\sqrt{2}}{2}, \pm\frac{\sqrt{2}}{2})$ with  the  function $h^{(\hat x)}(\eta)$ in Figure \ref{fig:3-5}.  The pulse moment $t_0$ is set to be $2,4,6$, respectively. The intervals $[\eta_1,\eta_2]$ such that $h^{(\hat x)}(\eta)\geq0$ are $[1,3], [3,5], [5,7]$ for $t_0=2,4,6$. Naturally, we obtain $t_0^{num}=2,4,6$. This illustrates that the algorithm proved to be effective for data collected from any arbitrary observation direction.

\begin{figure}[H]
\centering
\subfigure[$\eta=2$]{
\includegraphics[scale=0.22]{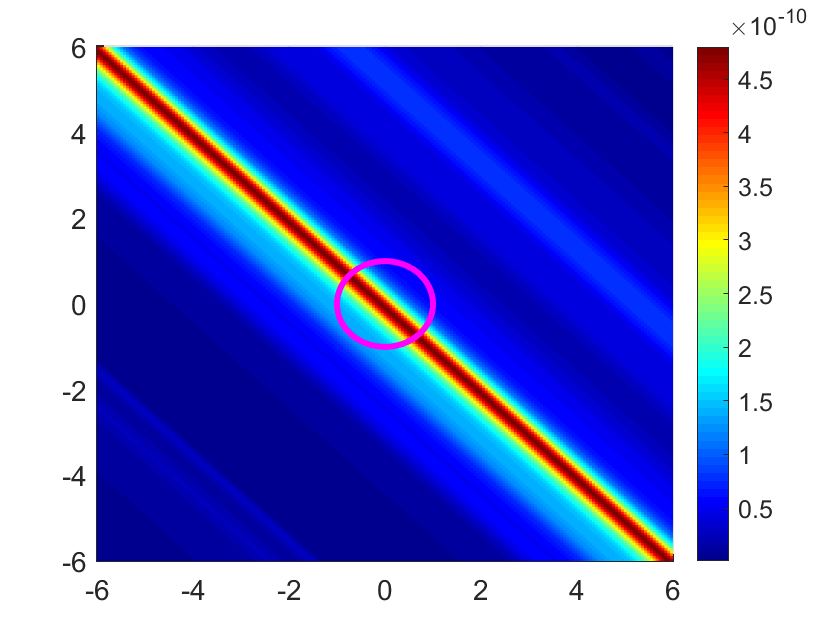}
}
\subfigure[$\eta=3$ ]{
\includegraphics[scale=0.22]{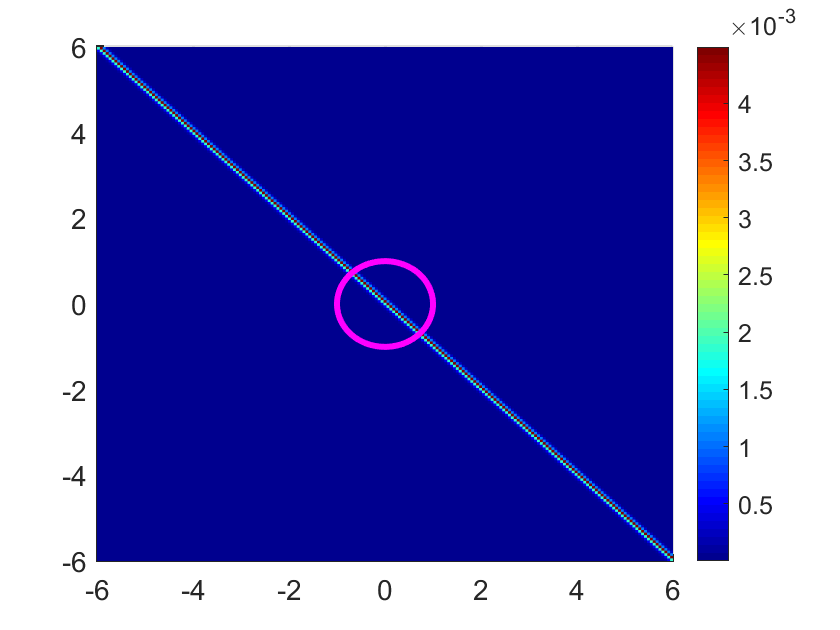}

}
\subfigure[$\eta=3.5$]{
\includegraphics[scale=0.22]{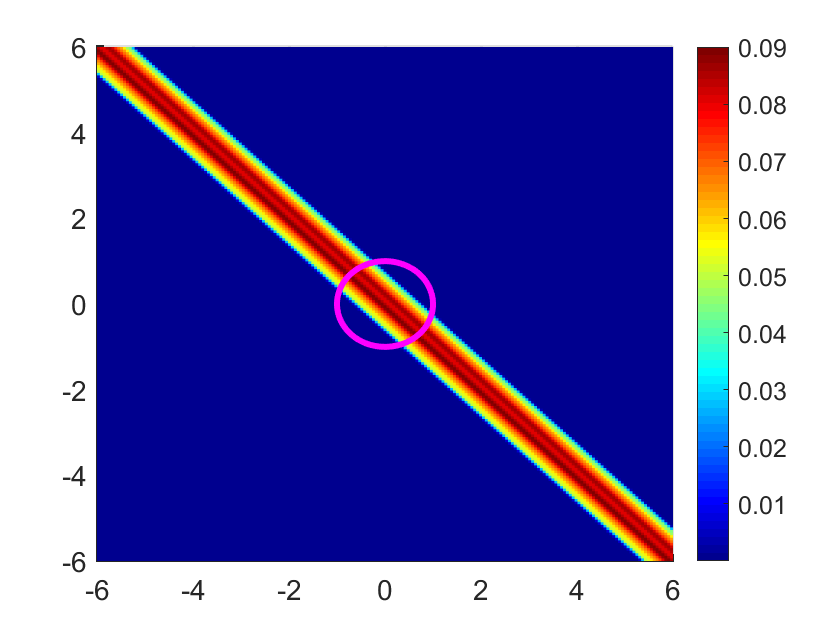}

}
\subfigure[$\eta=4$]{
\includegraphics[scale=0.22]{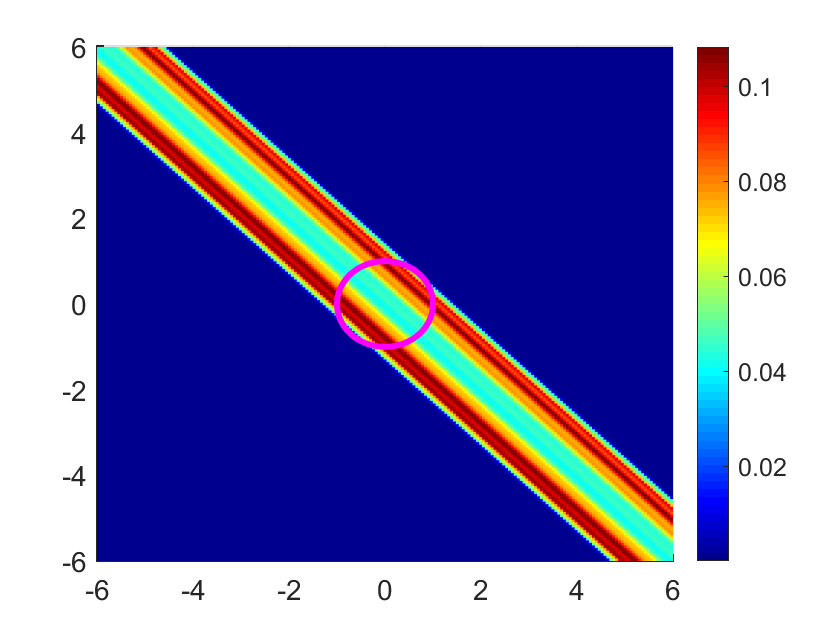}
}
\subfigure[$\eta=4.5$ ]{
\includegraphics[scale=0.22]{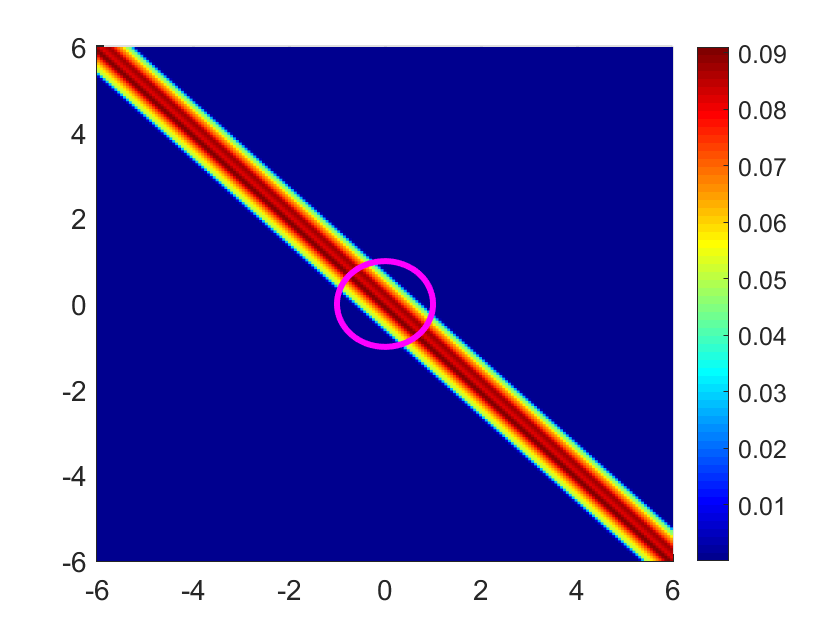}

}
\subfigure[$\eta=5$]{
\includegraphics[scale=0.22]{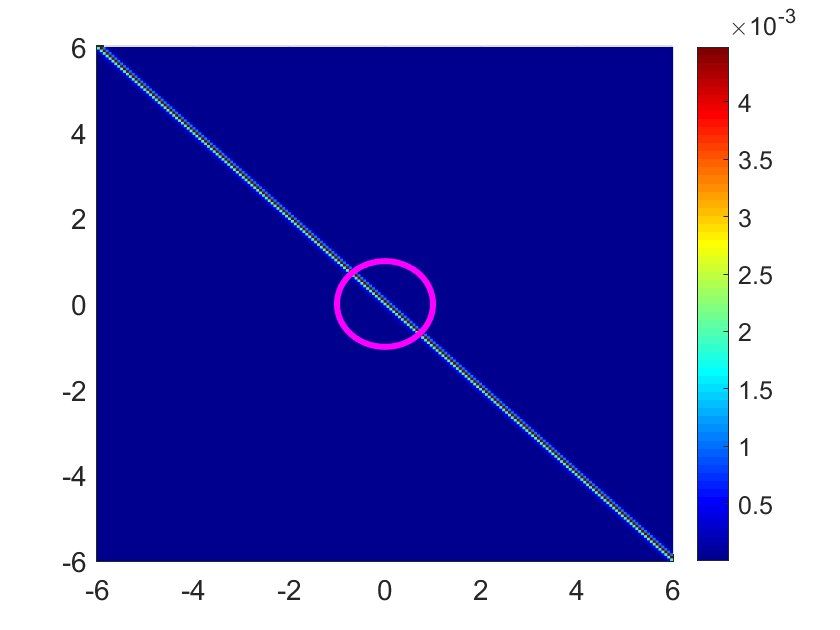}

}
\caption{Reconstruction for $K_{D,\eta}^{(\hat{x})} \cap K_{D,\eta}^{(-\hat{x})}$  using multi-frequency far-field data from a pair of observation directions $\hat x=(\pm\frac{\sqrt{2}}{2}, \pm\frac{\sqrt{2}}{2})$ with  the  indicator function $W_\eta^{(\hat x)}(y)$. Various $\eta$ are tested and the pulse moment is $t_0=4$.
} \label{fig:3-4}
\end{figure}

\begin{figure}[H]
\centering
\subfigure[$t_0=2$]{
\includegraphics[scale=0.22]{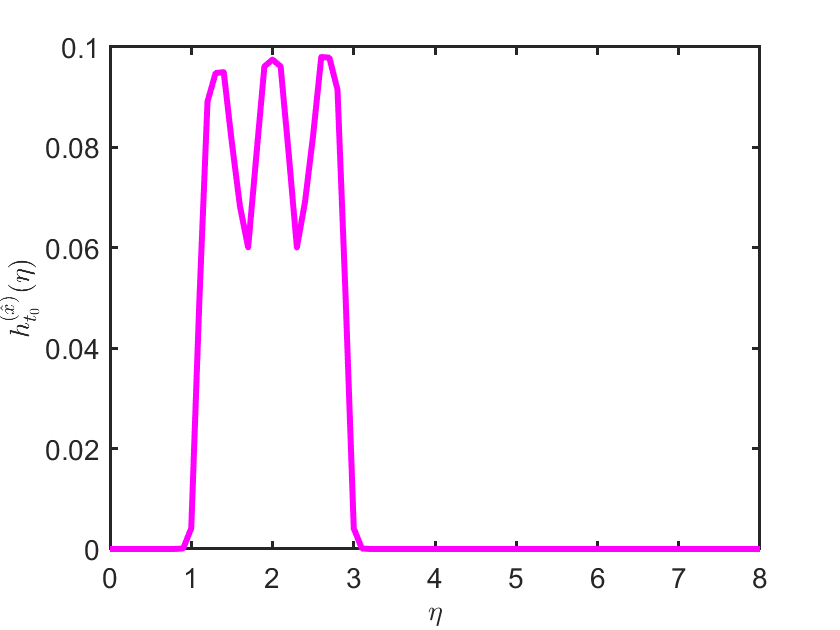}
}
\subfigure[$t_0=4$ ]{
\includegraphics[scale=0.22]{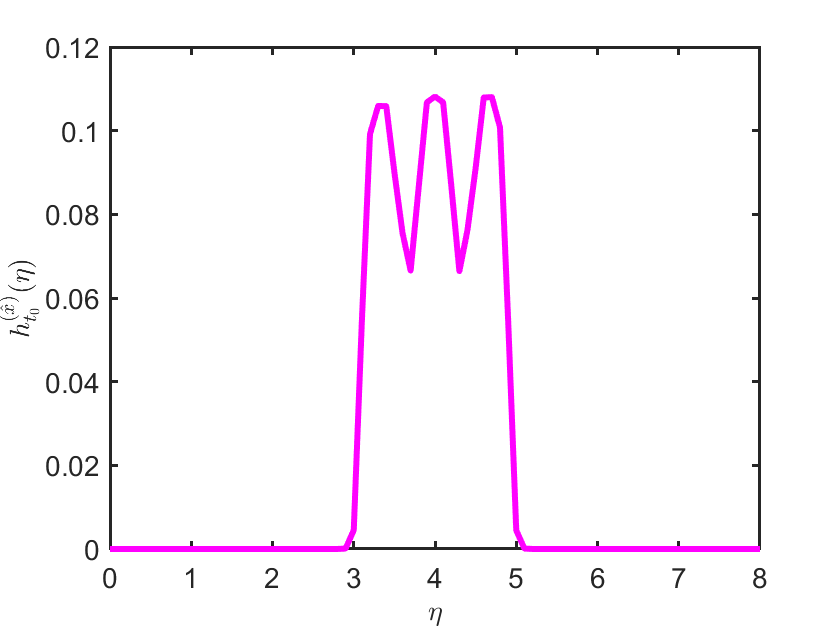}

}
\subfigure[$t_0=6$]{
\includegraphics[scale=0.22]{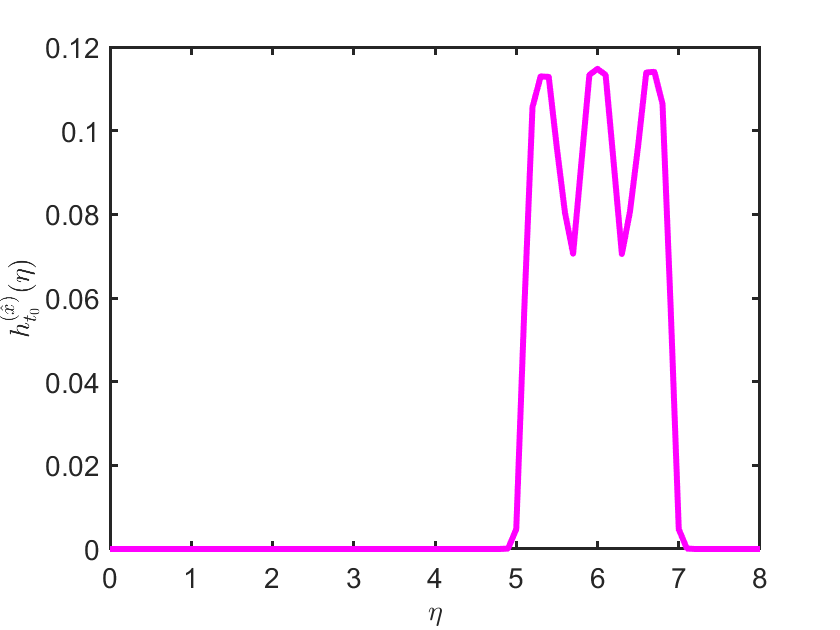}

}
\caption{Determination of the pulse moment $t_0$  using multi-frequency far-field data from a pair of observation directions $\hat x=(\pm\frac{\sqrt{2}}{2}, \pm\frac{\sqrt{2}}{2})$ with  the  function $h^{(\hat x)}(\eta)$.  The pulse moment $t_0$ is set to be $2,4,6$, respectively.
} \label{fig:3-5}
\end{figure}

\subsection{Reconstruction of the source  location and shape at time $t_0$}

Once the pulse moment  of the source has been reconstructed using multi-frequency far-field data from any observation direction, then we can set $t_0=\eta$ in the test function. According to Theorems \ref{Th:kd} and \ref{TH:hull}, by employing measurement data from one single or multiple direction(s), one can identify the strip $K_{D}^{(\hat{x})}$ and $\Theta$-convex hull of the source support that means the position and shape of the source at the pulse moment.  In the examples below, we consider the reconstructions of source supports with different shapes, each excited at different times. Specifically, we analyze a circular support excited at time $t_0=3$, a kite-shaped support excited at time $t_0=4$, and a rounded square support excited at time $t_0=5$. Figure \ref{fig:3-7} presents the determination of the pulse moment $t_0$  using multi-frequency far-field data from a pair of observation directions $\hat x=(0, \pm1)$ by plotting   the function $h^{(\hat x)}(y)$.  Figure \ref{fig:3-7} illustrates the non-zero intervals $[\eta_1,\eta_2]$ of the function $\eta \rightarrow h^{(\hat x)}(y)$ for different excitation times. This allows us to accurately determine the excitation times  $t_0^{num}=\frac{\eta_1+\eta_2}{2}$=3,4,5.  We show reconstructions using multi-frequency far-field data from $8$ observation directions for a circular, a kite-shaped and a round-square-shaped support with different pulse moment by plotting the indicator function  \eqref{indicator7} in Figure \ref{fig:3-6}. The shape of source support $D$ are highlighted by the pink solid line.  It is evident that both the shapes and locations are accurately reconstructed.

\begin{figure}[H]
\centering
\subfigure[$t_0=3$]{
\includegraphics[scale=0.22]{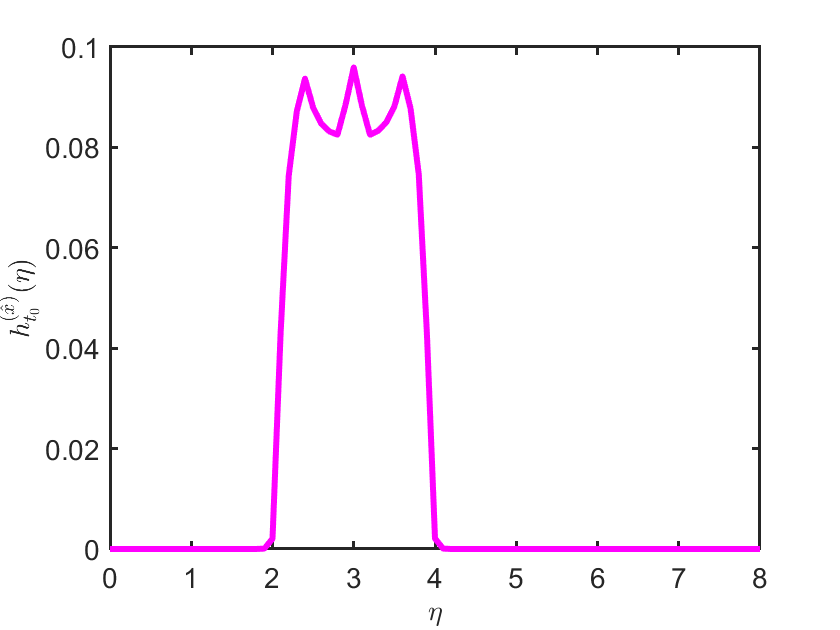}
}
\subfigure[$t_0=4$ ]{
\includegraphics[scale=0.22]{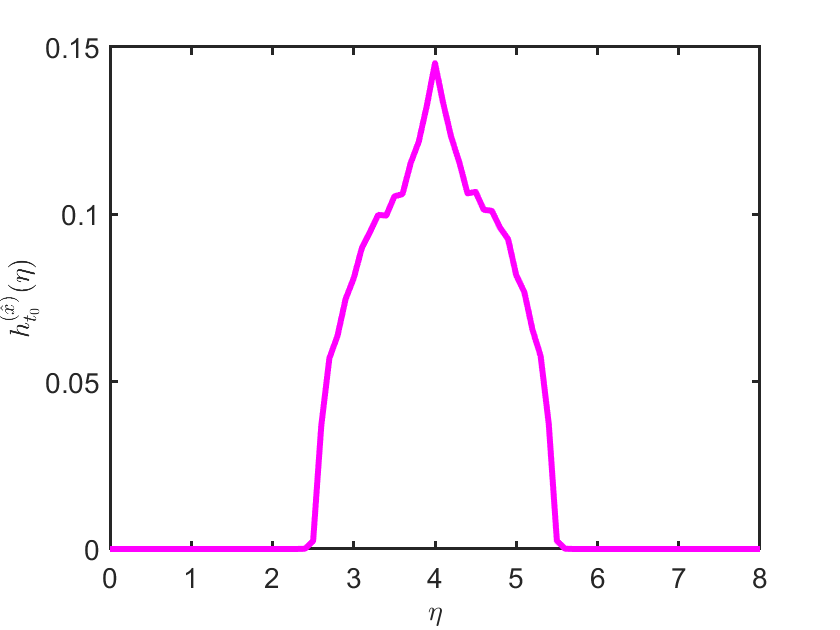}

}
\subfigure[$t_0=5$]{
\includegraphics[scale=0.22]{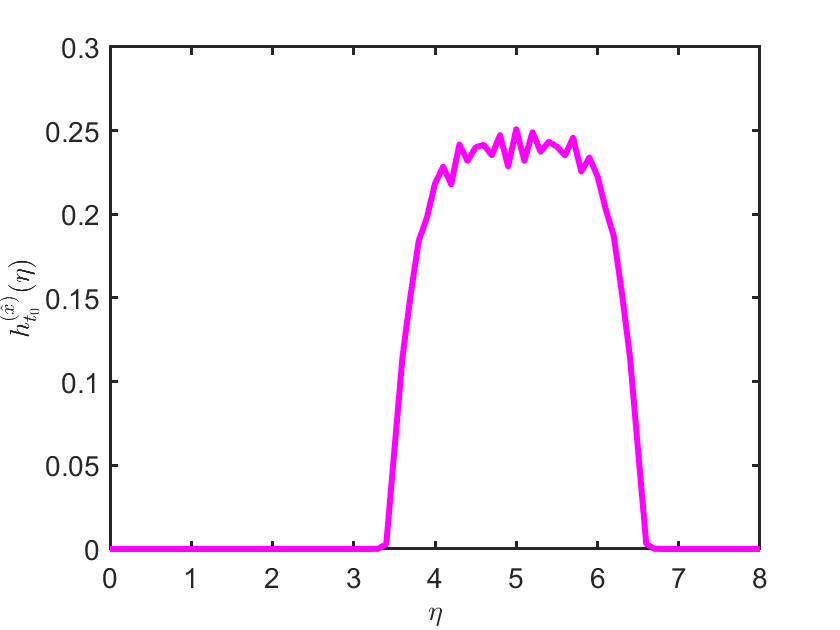}

}
\caption{Determination of the pulse moment $t_0$  using multi-frequency far-field data from a pair of observation directions $\hat x=(0, \pm1)$ with  the  function $h^{(\hat x)}$.  The pulse moment $t_0$ is set to be $3,4,5$ for a circular support,a kite-shaped support and a round-square-shaped support, respectively.
} \label{fig:3-7}
\end{figure}

\begin{figure}[H]
\centering
\subfigure[$t_0=3$]{
\includegraphics[scale=0.22]{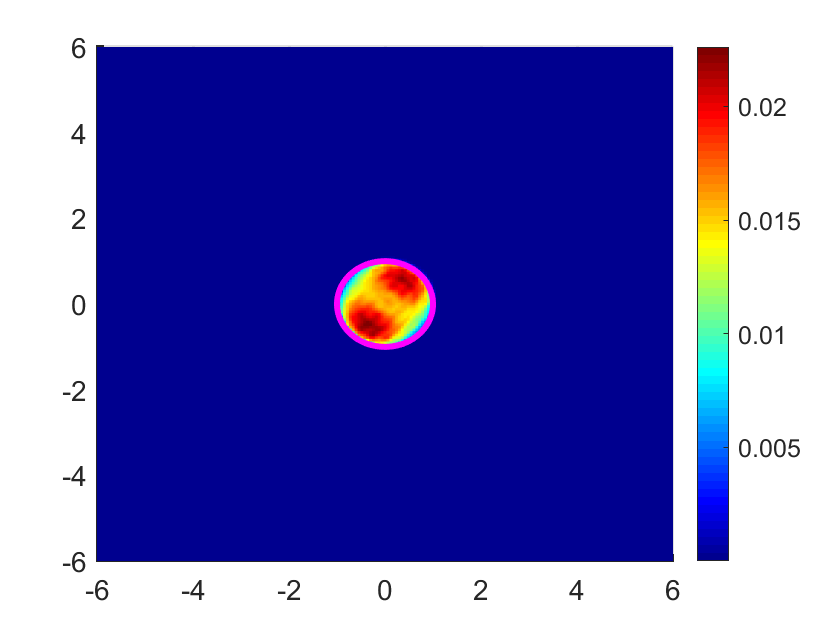}
}
\subfigure[$t_0=4$ ]{
\includegraphics[scale=0.22]{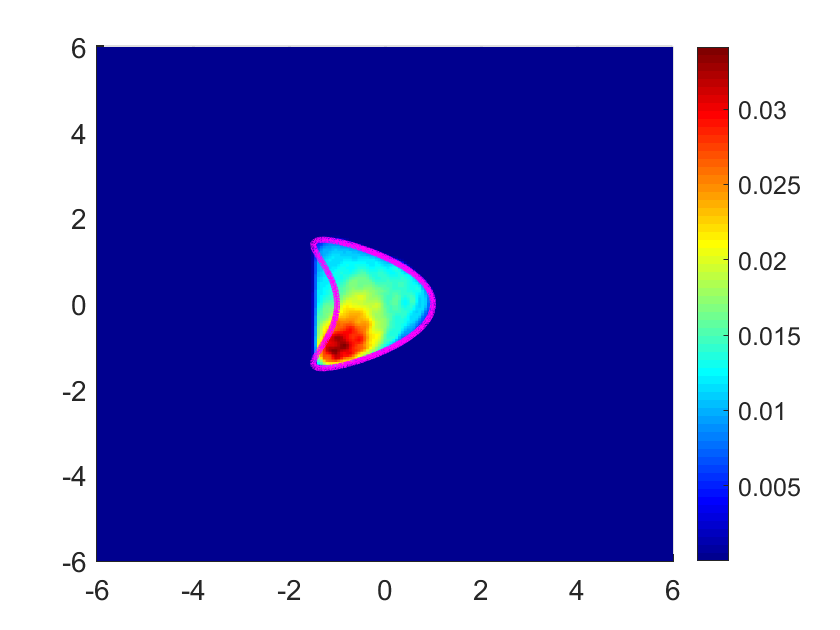}

}
\subfigure[$t_0=5$]{
\includegraphics[scale=0.22]{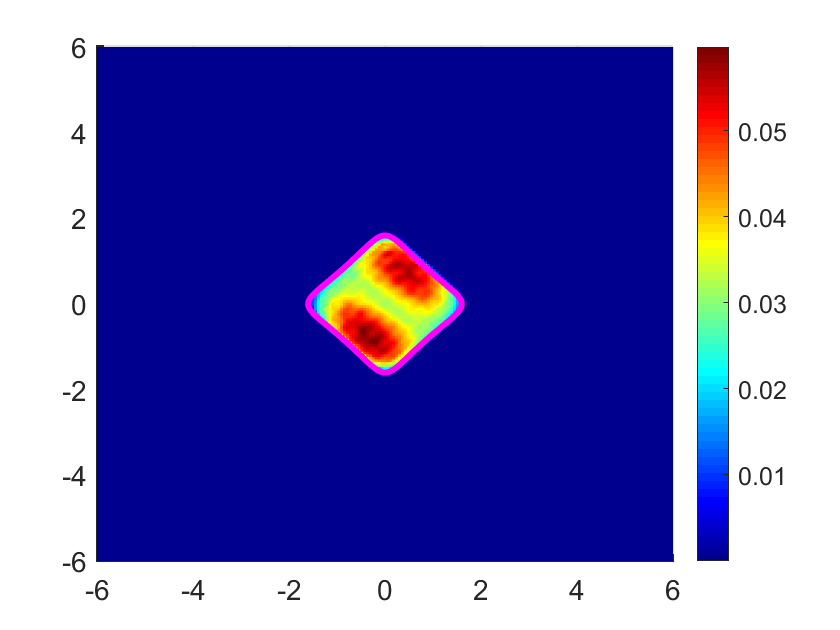}

}
\caption{Reconstructions using multi-frequency far-field data from $8$ observation directions for a circular, a kite-shaped and a round-square-shaped support with different pulse moments.
} \label{fig:3-6}
\end{figure}

\subsection{Reconstruction of the trajectory of a moving source }

In this section, we focus on the reconstruction of the trajectory of a moving extended source.  By utilizing multi-frequency far-field  data collected from finite observation directions, we aim to accurately determine the trajectory followed by the potential over time. To enhance imaging efficacy, the indicator function values in this subsection are normalized to their corresponding maximum values. We mainly consider three distinct trajectories for the moving potential: the cardioid, the trifolium rose curve, and the star curve. These trajectories are parameterized as follows:
\begin{itemize}
	\item  $a_1(t)=8 \sin^3 \frac{\pi}{40}t $, $a_2(t)=6\cos \frac{\pi}{40}t- 2 \cos \frac{2\pi}{40}t- \cos \frac{3\pi}{40}t- 0.5 \cos \frac{4\pi}{40}t$, $t\in[0,80)$,
	\item   $a_1(t)=8\cos \frac{3\pi}{40}t \cos \frac{\pi}{40}t$, $a_2(t)=8\cos \frac{3\pi}{40}t  \sin \frac{\pi}{40}t $, $t\in[0,80)$,
	\item  $a_1(t)=7\cos^3 \frac{\pi}{40}t $, $a_2(t)=7\sin^3 \frac{\pi}{40} t$, $t\in[0,80)$.
\end{itemize}

\noindent The source supports are supposed to be $|x-a(t_j)| \leq 0.1$  for each impulse moment $t_j, j=0,2,...,J$. For the sake of simplicity we suppose that the source profile remain invariant along the time.
 The search region is chosen as $[-10,10]\times[-10,10]$. The multi-frequency data are collected from $8$ observation directions.
In Figure \ref{fig:3-8}(a), we undertake the reconstruction of the trajectory for the moving extended source, where each time point is represented by $t_j=2j, j=0,...,39$. Additionally, we display the original trajectory of the moving extended source $a(t_j), j=0,...,39$ in Figure \ref{fig:3-8}(d). It is evident that the cardioid trajectory is faithfully reconstructed. Similarly, we conducted reconstructions of the  trajectories for the  trifolium rose curve and star curve, depicted in Figures \ref{fig:3-8}(b) and(c), respectively. The original trajectories are showcased in Figures \ref{fig:3-8}(e) and(f). Notably, the reconstructions exhibit remarkable fidelity to the original trajectories, underscoring the efficacy of the reconstruction process.

\begin{figure}[H]
\centering
\subfigure[Cardioid ]{
\includegraphics[scale=0.22]{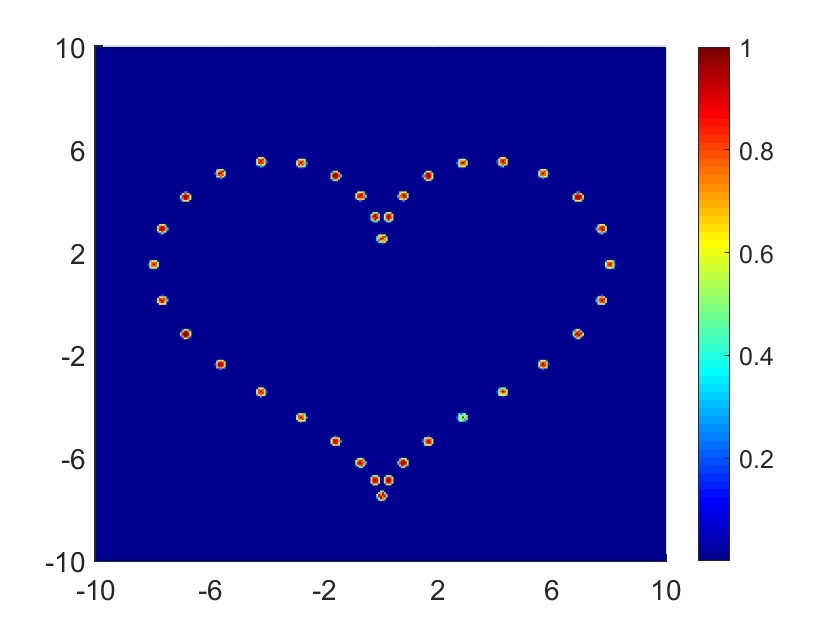}
}
\subfigure[Trifolium rose curve]{
\includegraphics[scale=0.22]{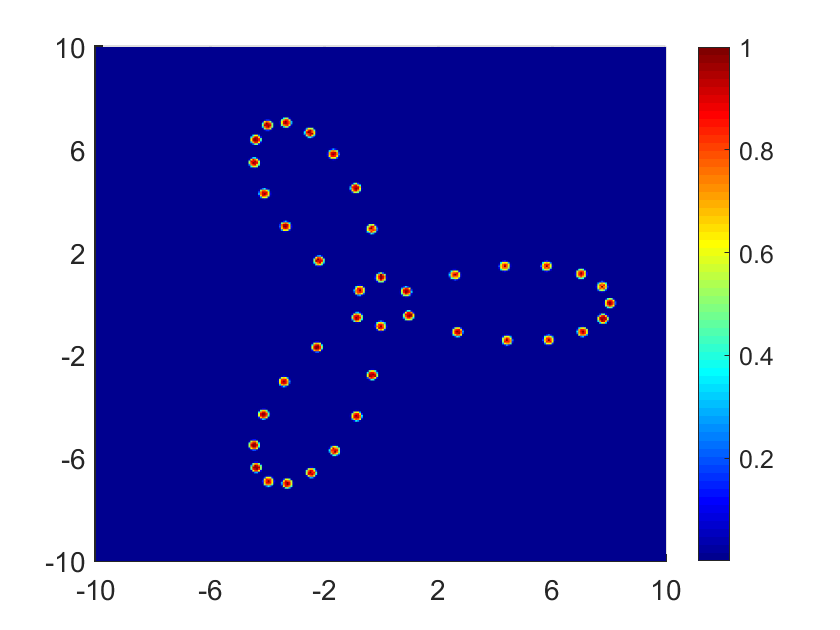}

}
\subfigure[Star Curve]{
\includegraphics[scale=0.22]{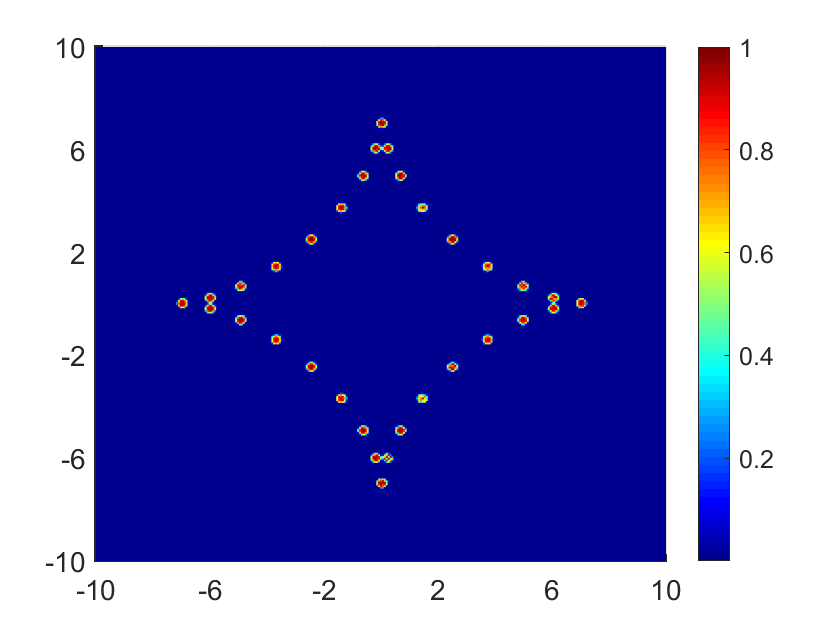}

}
\subfigure[Original cardioid ]{
\includegraphics[scale=0.22]{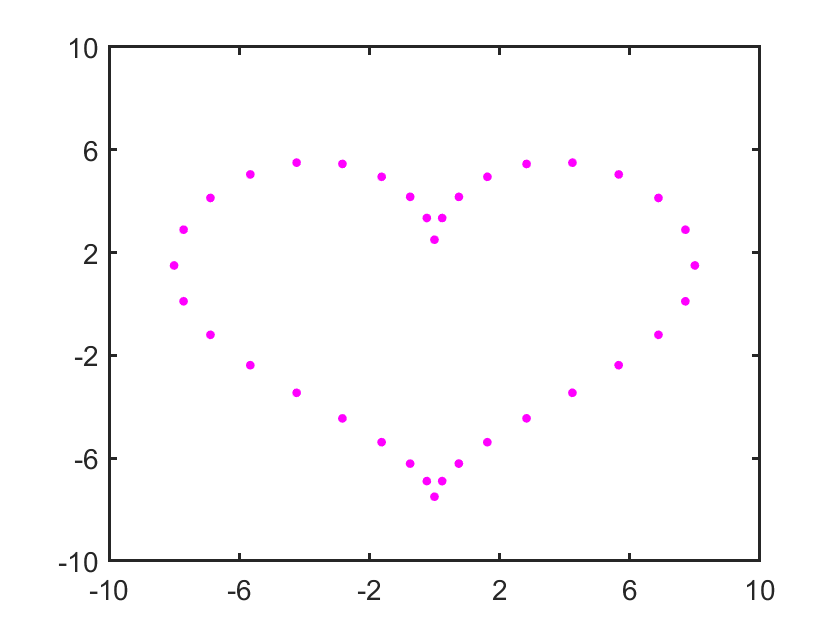}
}
\subfigure[Original trifolium rose curve]{
\includegraphics[scale=0.22]{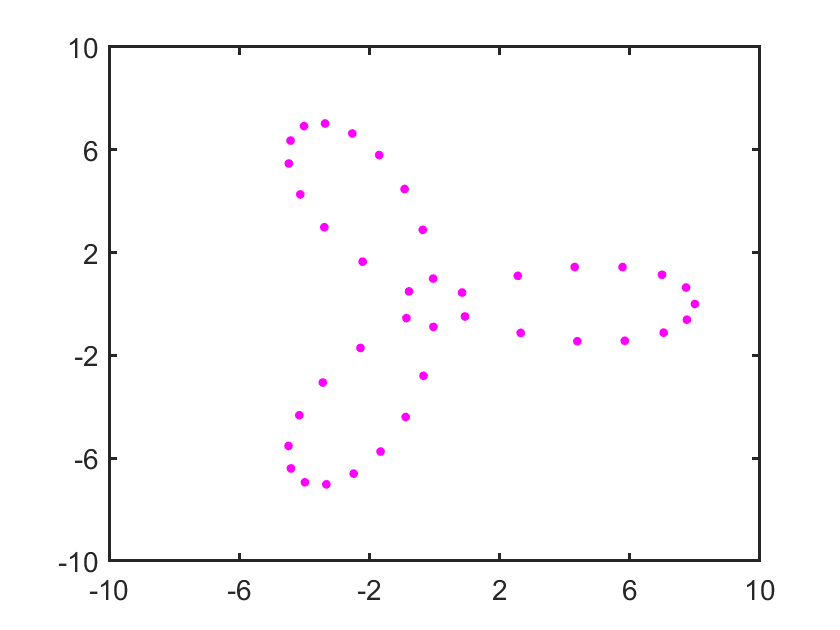}

}
\subfigure[Original star Curve]{
\includegraphics[scale=0.22]{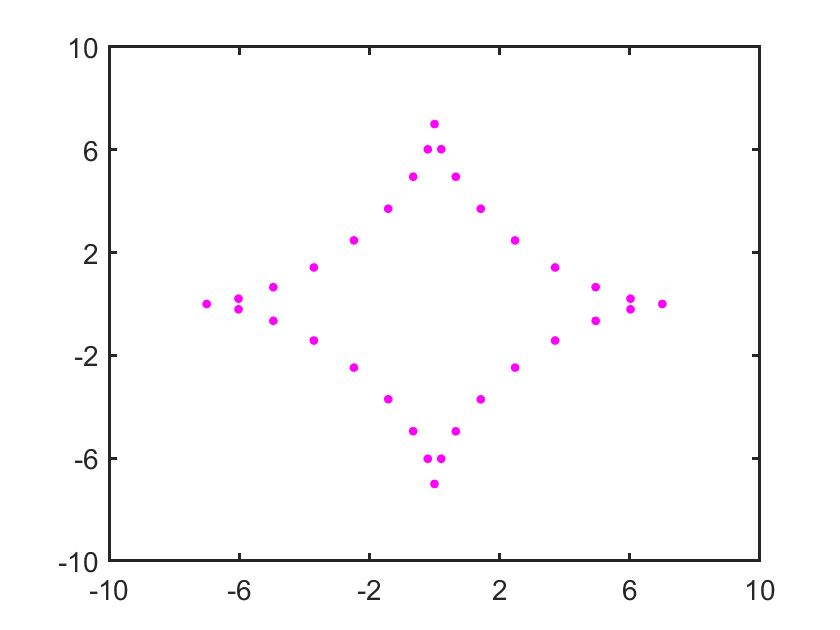}

}

\caption{Reconstruction using multi-frequency far-field data from 8 observation directions for the trajectory of the  extended source moving along three different curves. The first row illustrates the reconstructed motion trajectories of the moving extended source, the second row depicts the original motion trajectories.
} \label{fig:3-8}
\end{figure}

Furthermore, we consider a moving kite, the diameter of the kite is $3$. Hence, the impulse moments are set $t_j=0,4,8,12,16$. The motion trajectory of the center of the kite is $a(t)=(t-8, 4*\sin(\frac{\pi}{8}(t-8)+\pi))$.

\begin{figure}[H]
\centering
\subfigure[$t_0=0$ ]{
\includegraphics[scale=0.12]{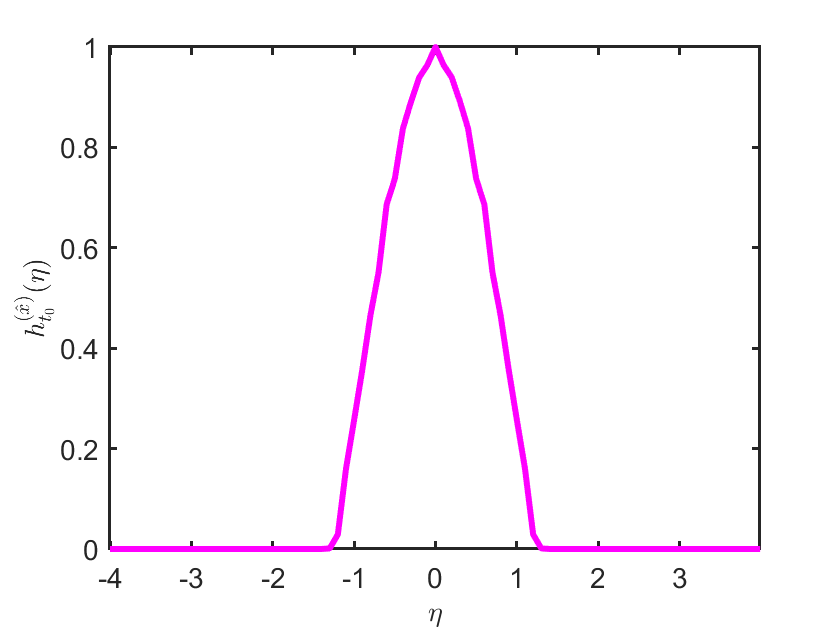}
}
\subfigure[$t_1=4$]{
\includegraphics[scale=0.12]{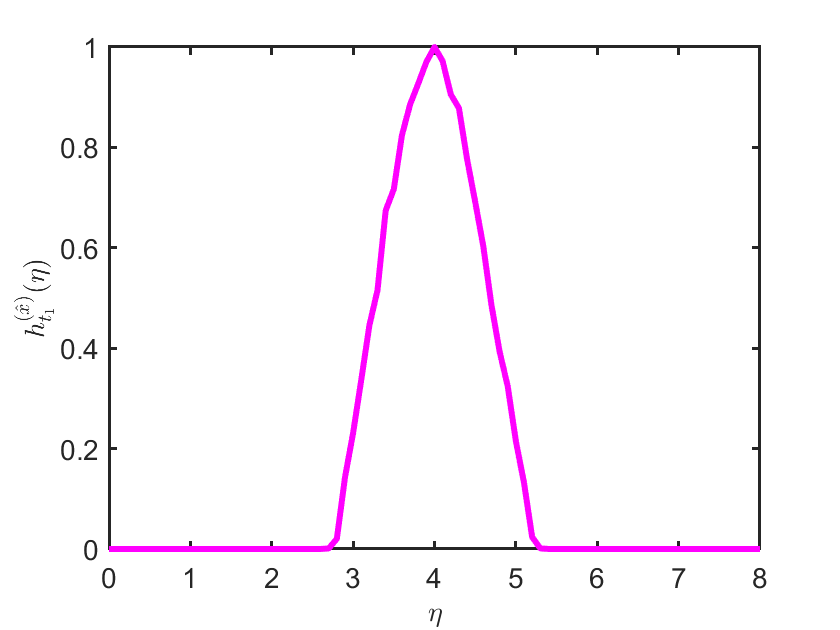}
}
\subfigure[$t_2=8$ ]{
\includegraphics[scale=0.12]{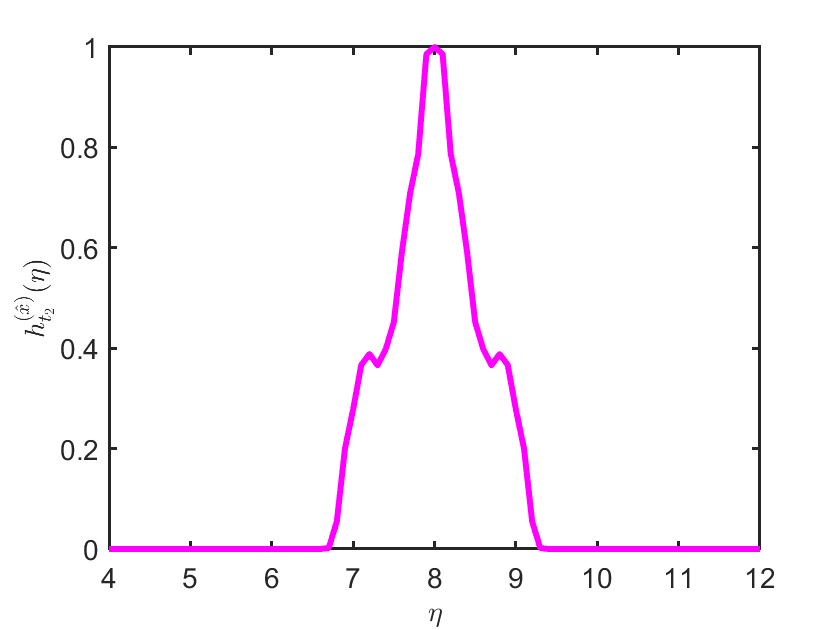}
}
\subfigure[$t_3=12$ ]{
\includegraphics[scale=0.12]{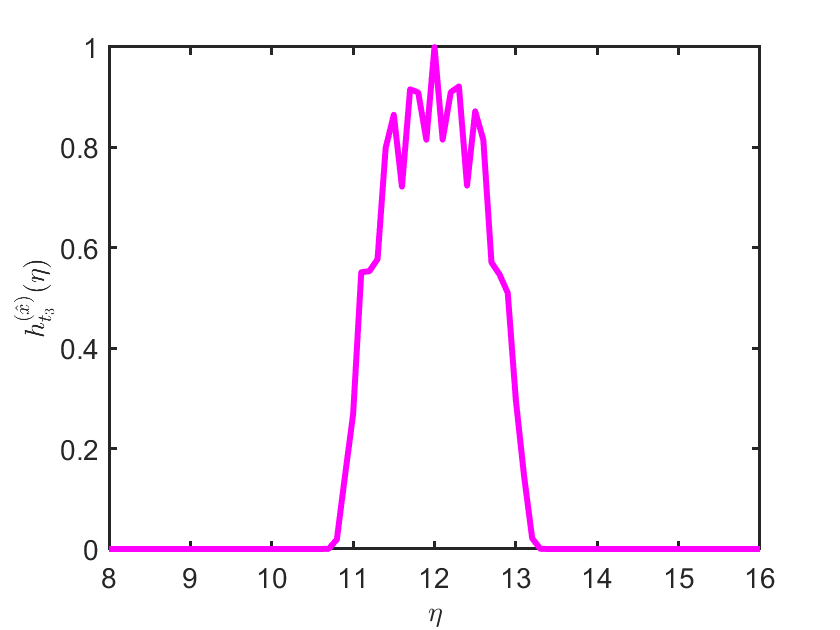}
}
\subfigure[$t_4=16$ ]{
\includegraphics[scale=0.12]{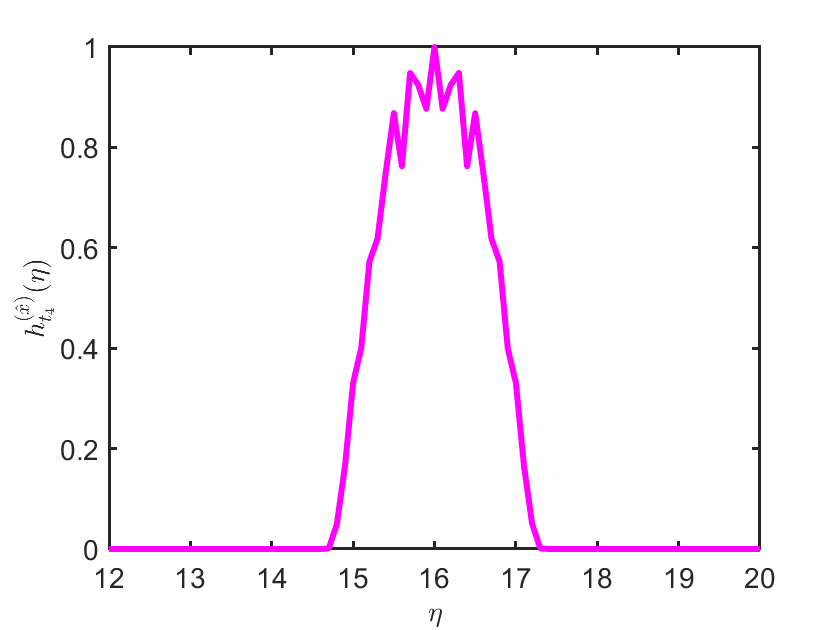}
}
\caption{Reconstruction for the trajectory of a moving kite along a sin function line.
} \label{fig:3-9}
\end{figure}

\begin{figure}[H]
\centering
\subfigure[Original trajectory]{
\includegraphics[scale=0.3]{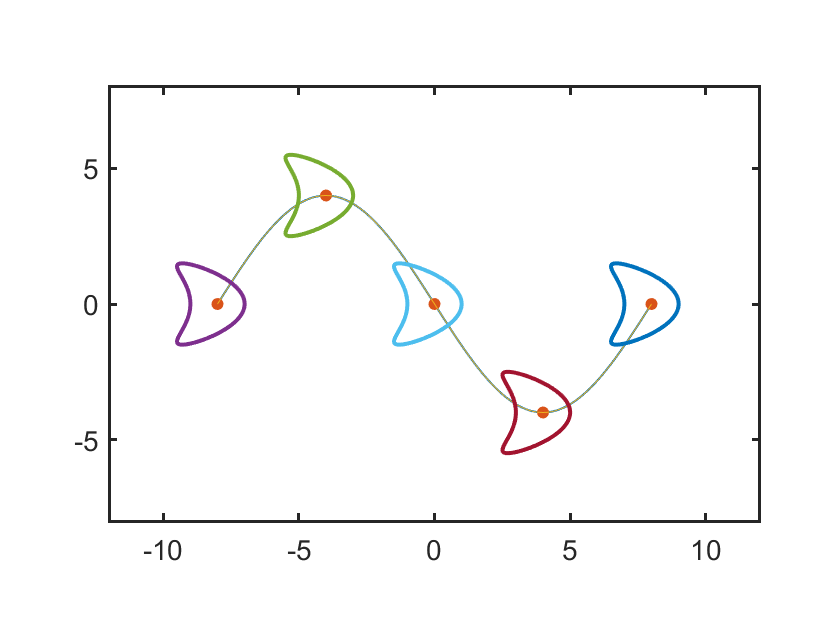}
}
\subfigure[Reconstruction]{
\includegraphics[scale=0.3]{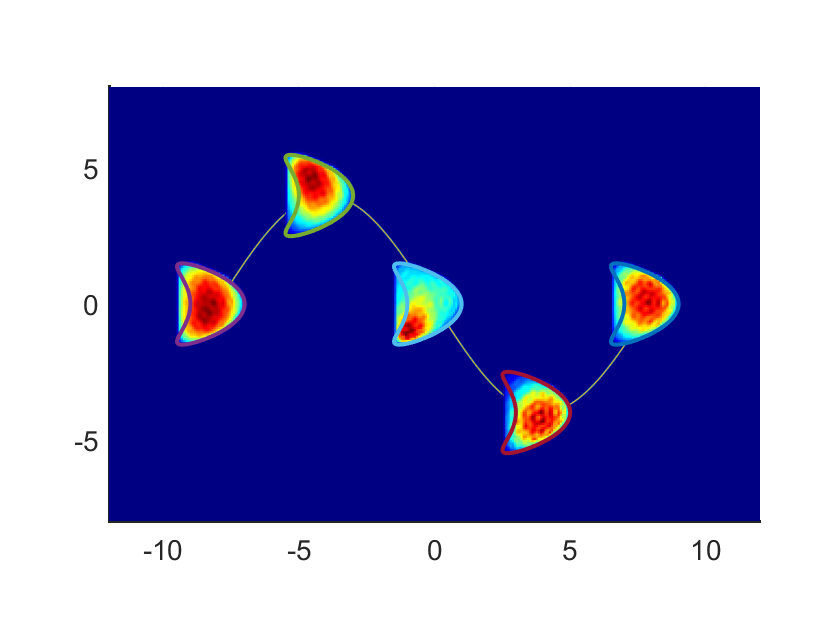}

}
\caption{Reconstruction for the trajectory of a moving kite along a sin function line.
} \label{fig:3-9}
\end{figure}

\subsection{Noise test}

We evaluate sensitivity with respect to the noisy data by selecting the numerical example in Figure \ref{fig:3-8}(a), which involves  cardioid recovery. The far-field data are polluted by Gaussian noise, as shown below:
\begin{equation}
F_{\delta}^{(\hat x)}:=F^{(\hat x)}+\delta \| F^{(\hat x)}\|_2 \mathcal M,
\end{equation}
where the data matrix $F^{(\hat x)}$ is defined by (\ref{matF}), $\delta$ represents the noise level and $\mathcal M \in \R^{N\times N}$ is a uniformly distributed random -matrix with the
random variable ranging from $-1$ to $1$.

In Figure \ref{fig:3-9}, the noise levels tested are $2\%$, $5\%$ and $10\%$.  The  the inversion algorithm's performance is evaluated with a noise level of $2\%$ In Figure \ref{fig:3-9}(a). The reconstructed image shows minor artifacts, indicating that the algorithm is relatively robust to low levels of noise.  Figure \ref{fig:3-9}(b) demonstrates the inversion results when the noise level is increased to $5\%$. The image quality slightly deteriorates compared to the $2\%$ noise level, with more noticeable artifacts appearing in the reconstructed image. The results depicted in Figure \ref{fig:3-9}(a) are obtained with a noise level of $10\%$. As expected, the higher noise level introduces significant artifacts, making the reconstructed trajectory  almost invisible.

\begin{figure}[H]
\centering
\subfigure[$\delta=0\%$ ]{
\includegraphics[scale=0.16]{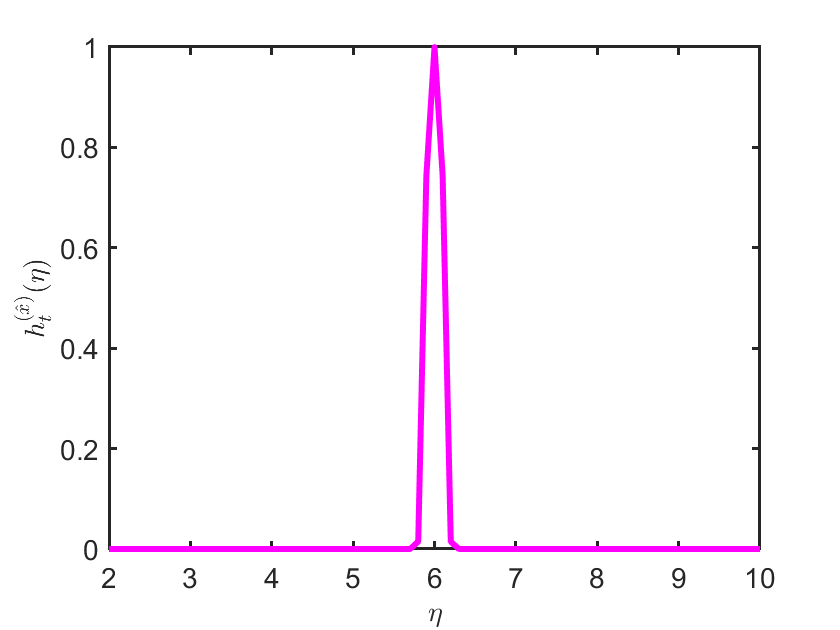}
}
\subfigure[$\delta=2\%$]{
\includegraphics[scale=0.16]{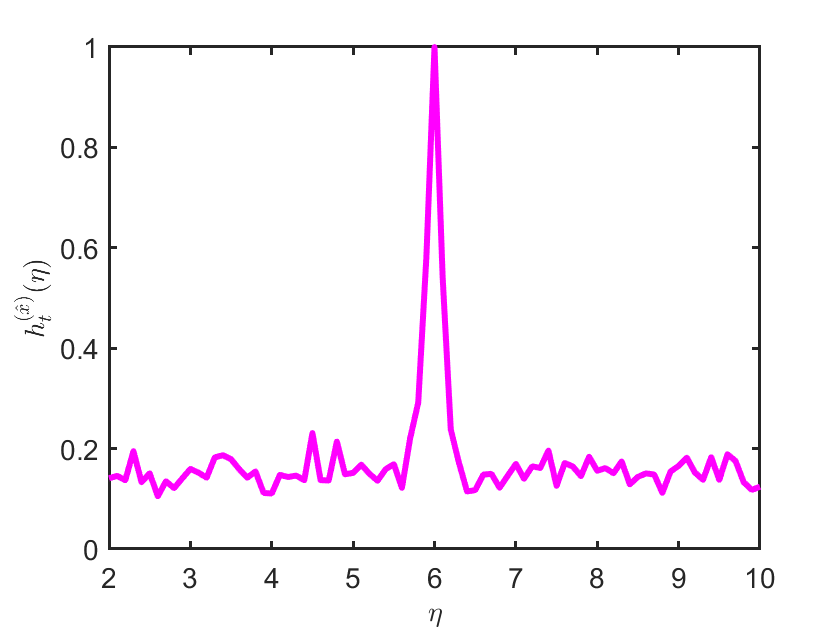}

}
\subfigure[$\delta=5\%$]{
\includegraphics[scale=0.16]{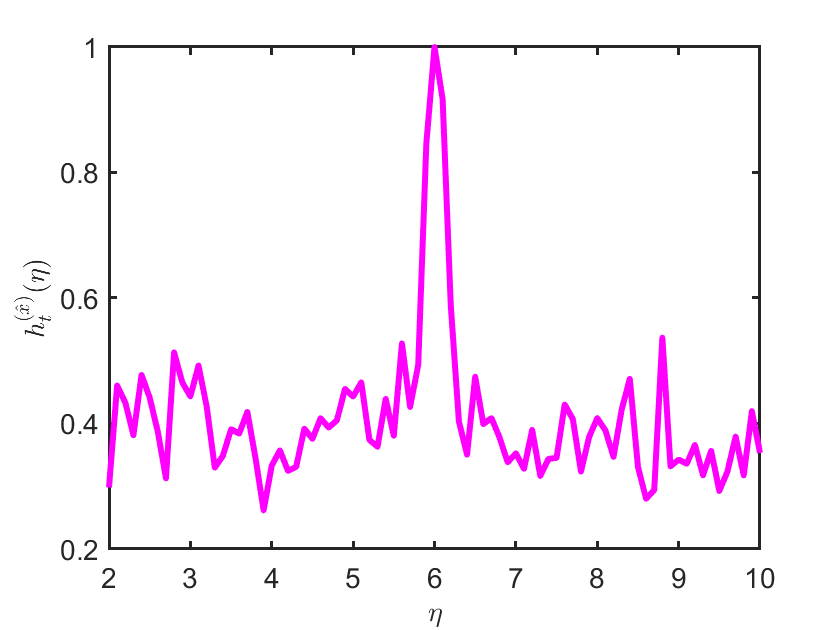}

}
\subfigure[$\delta=10\%$]{
\includegraphics[scale=0.16]{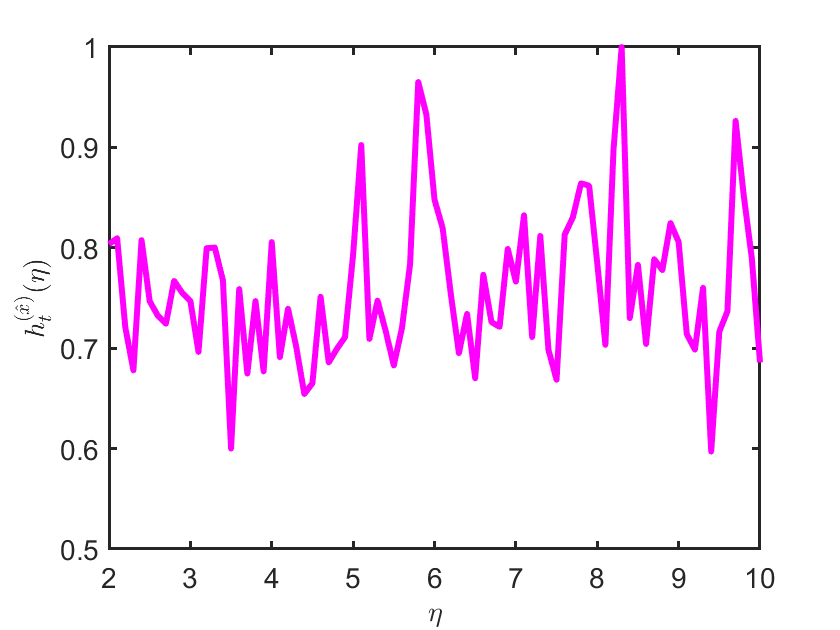}

}
\caption{Reconstruction for the impulse moment $t_j=6$ at different noise levels $\delta$.
} \label{fig:3-9}
\end{figure}

\begin{figure}[H]
\centering
\subfigure[$\delta=2\%$ ]{
\includegraphics[scale=0.22]{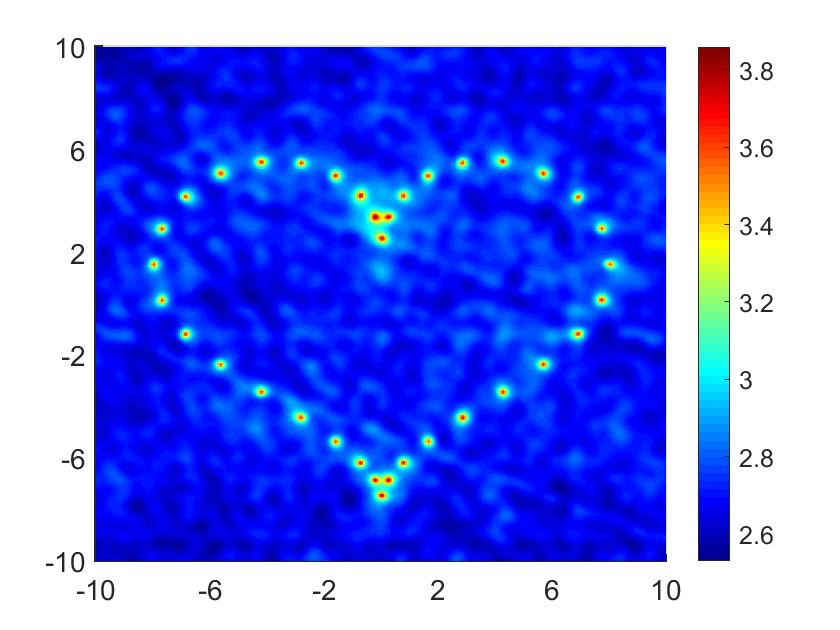}
}
\subfigure[$\delta=5\%$]{
\includegraphics[scale=0.22]{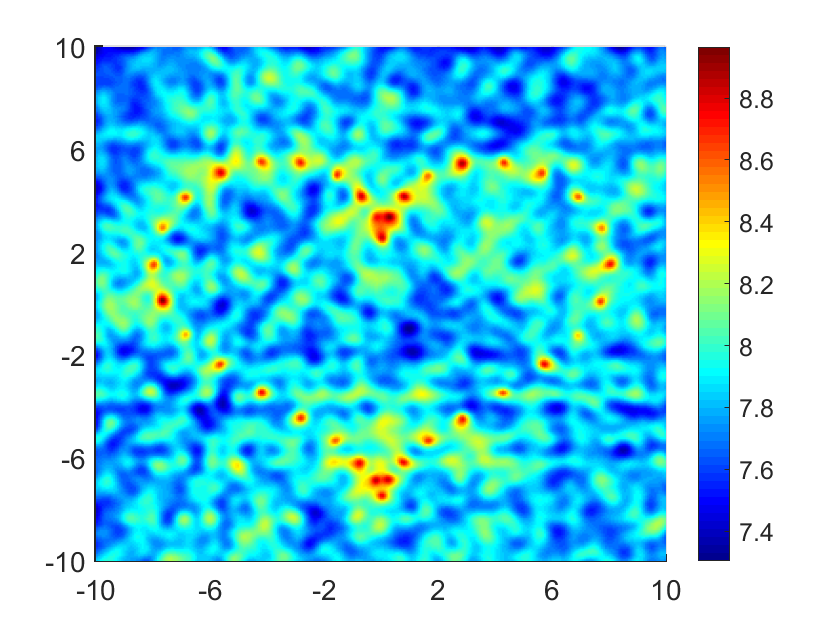}

}
\subfigure[$\delta=10\%$]{
\includegraphics[scale=0.22]{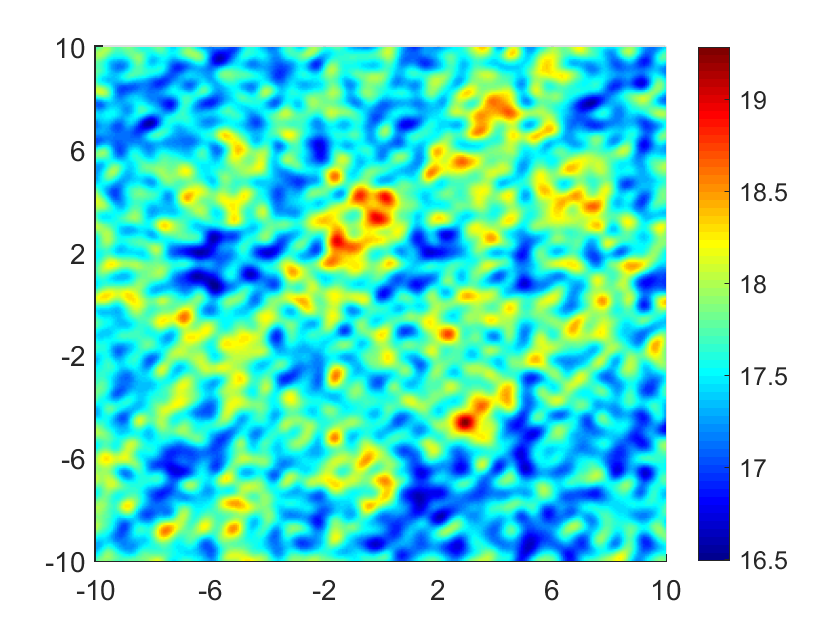}

}

\caption{Reconstruction for the trajectory of a extended source moving along cardioid using observation data at different noise levels $\delta$.
} \label{fig:3-9}
\end{figure}

%\subsection{Reconstruction  in $\R^3$}

%We will show a helices in $\R^3$.

\subsection{Reconstruction of a moving extended source in $\R^3$}

In this section, we turn our attention to the reconstruction of the trajectory of  a moving extended source in  $\R^3$, specifically demonstrating a helical trajectory parameterized as parameterized as $a(t)=(\cos \frac{1}{4}t, \frac{1}{12}t-5, \sin \frac{1}{4}t)$ for $t\in [0,120]$. We assume the impulse moments are $t_j=2j, j=0,..,60$.  The source supports is constrained by $|x-a(t_j)|\leq 0.1$ for each impulse moment $t_j$. To encompass the entire trajectory of the moving extended  source,
the testing domain is selected to be  $[-6,6]^3$. Given the relatively small source support, the domain is   discretized with a resolution of $150$ points in each dimension. The multifrequency data are collected from these $10$ observation directions.
Figure \ref{fig:3d-1} presents the reconstruction for the helical trajectory of the moving extended source using multifrequency data from $10$ observation directions uniformly distributed on the upper unit hemisphere. Figure \ref{fig:3d-1}(a) features the original helical depicted by a green solid line  with pink small spheres representing  the locations ans shapes at each pulse moment $t_j$.  Figure \ref{fig:3d-1}(b) shows the reconstructions of the moving extended source. For comparative purpose, the original trajectory is superimposed  with a green solid line.
It is evident  that the location and shape of the moving extended source for each pulse moment $t_j$ are accurately reconstructed, despite their shapes not being as spherical as a ball. This is primarily due to our selection of only 10 observation directions, leading to  iso-surfaces that inevitably deviate from a perfect sphere.  The reconstructions for the helix trajectory in $\R^3$ further verify the accuracy of our algorithm.

\begin{figure}[H]
\centering
\subfigure[Original trajectory]{
\includegraphics[scale=0.3]{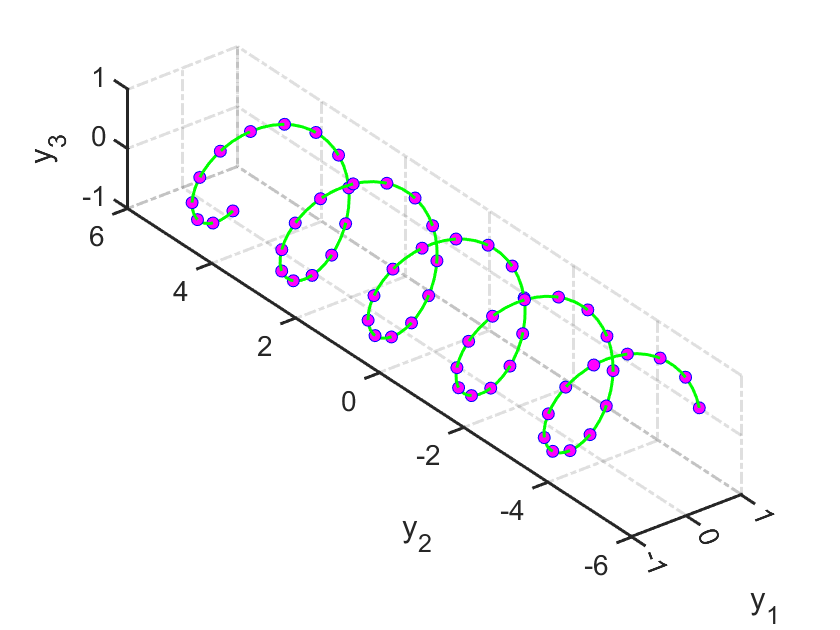}
}
\subfigure[Reconstructed trajectory]{
\includegraphics[scale=0.3]{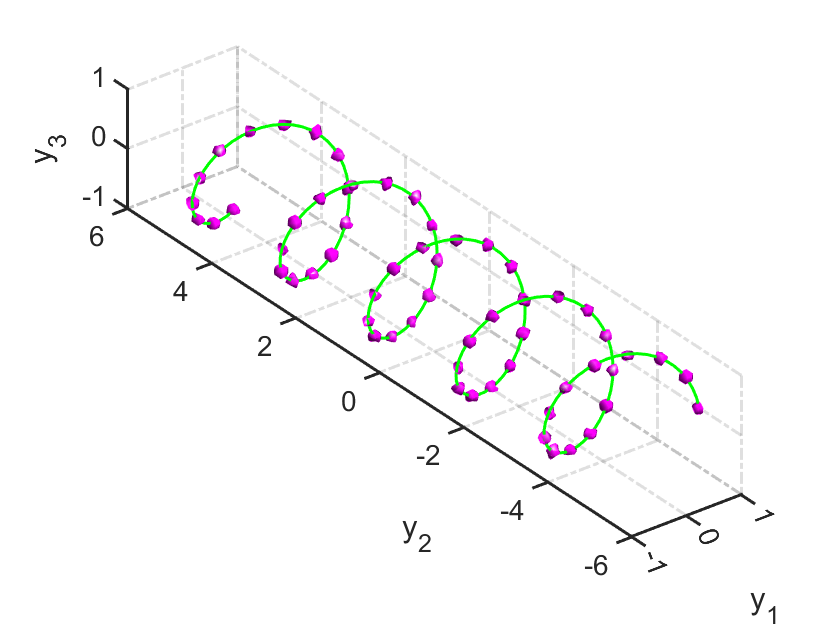}

}\caption{
Reconstruction using multi-frequency far-field data from 10 observation directions for the trajectory of the extended source moving along a helix.} \label{fig:3d-1}
\end{figure}

% \begin{figure}[H]
% \centering
% \subfigure[$\delta=2\%$ ]{
% \includegraphics[scale=0.3]{Fig/Fig3/6-3d.png}
% }
% \subfigure[$\delta=5\%$]{
% \includegraphics[scale=0.3]{Fig/Fig3/5-3d.png}

% }\caption{Reconstruction for the trajectory of a extended source moving along cardioid using observation data at different noise levels $\delta$.
% }
% \end{figure}

% \begin{figure}[H]
% \centering
% \includegraphics[scale=0.6]{Fig/Fig3/5-1.png}
% \caption{Reconstruction for the trajectory of a extended source moving along cardioid using observation data at different noise levels $\delta$.
% }
% \end{figure}

% \begin{figure}[H]
% \centering
% \subfigure[$\delta=2\%$ ]{
% \includegraphics[scale=0.3]{Fig/Fig3/5-1.png}
% }
% \subfigure[$\delta=5\%$]{
% \includegraphics[scale=0.3]{Fig/Fig3/5-1.png}
% }
% \caption{Reconstruction for the trajectory of a extended source moving along cardioid using observation data at different noise levels $\delta$.
% }
% \end{figure}

\section*{Acknowledgements}
The work of G. Hu is supported by the National Natural Science Foundation of China (No. 12425112), the Fundamental Research Funds for Central Universities in China (No. 050-63213025) and the Natural Science Foundation of Tianjin (No. 25JCZDJC00970).

\section*{Appendix}
{\mgqb
\newtheorem*{mytheorem}{Theorem A.1}
\begin{mytheorem}
\upshape {(Range Identity)} Let $X \subset U \subset X^*$ be a Gelfand triple with a space $U$ and a Hilbert space $Y$. Let $F: Y\rightarrow Y$, $L: X^*\rightarrow Y$, $T: X\rightarrow X^*$ be linear bounded operators such that $F=LTL^*$. We make the following assumptions
\begin{itemize}
	\item[(i)] $L$ is compact with dense range and thus $L^*$ is compact and one-to-one.
	\item[(ii)] $\real T$ is one-to-one, $\ima T=0$ or $\ima T$ is one-to-one, and the operator $T_{\#}=|\real T| +|\ima T|: X\rightarrow X^*$ is coercive, i.e., there exists $c>0$ with
	\begin{equation*}
	\big\langle T_{\#}\, \varphi, \varphi\big\rangle\geq c\,||\varphi||_X^2\quad\mbox{for all}\quad \varphi\in X,
	\end{equation*}
	where $\big\langle \cdot, \cdot\big\rangle$ represents the dual pairing between $X^*$ and $X$.
\end{itemize}
Then the operator $F_{\#}$ is positive and  the ranges of $F_{\#}^{1/2}:Y\rightarrow Y$ and  $L: X^*\rightarrow Y$ coincide.
\end{mytheorem}
If $\ima T=0$, the result of Theorem A.1 was justified by \cite[Theorem 2.15]{KG08}. If $\ima T \neq 0$, the proof of Theorem A.1 can be carry out following \cite[Theorem 2.2]{GGH2022}.
}

\newtheorem*{mtheorem}{Theorem A.2}
\begin{mtheorem}
\upshape {(Picard's theorem \cite[Theorem 4.8]{CK1998})}
Let $A:X\to Y$ be a compact linear operator with singular system $(\mu_n,\varphi_n,g_n)$. The equation of the first kind $A\phi =f$ is solvable if and only if $f$ belongs to the orthogonal complement $N(A^\ast)^{\perp}$ and satisfies
$$\sum\limits_{n=1}^{\infty}\frac{1}{\mu^2_n}|(f,g_n)|^2<\infty.$$
In this case a solution is given by
$$
\phi=\sum_{n=1}^\infty \frac{1}{\mu_n}(f,g_n)\,\varphi_n.
$$
\end{mtheorem}

%\end{thebibliography}


\begin{thebibliography}{00}

\bibitem{AHLS} A. Alzaalig, G. Hu, X. Liu and J. Sun, Fast acoustic source imaging using multi-frequency sparse data, Inverse Problems, 36 (2020): 025009.


\bibitem{BLLT} G. Bao, P. Li,  J. Lin and F. Triki, Inverse scattering problems with multi-frequencies, Inverse Problems, 31 (2015): 093001.

\bibitem{BLT10} G. Bao, J. Lin, and F. Triki, A multi-frequency inverse source problem, J. Differential Equations, 249 (2010): 3443--3465.

\bibitem{BLRX} G. Bao, S. Lu, W. Rundell and B. Xu, A recursive algorithm for multi-frequency acoustic inverse source
problems, SIAM J. Numer. Anal., 53 (2015): 1608-1628.

%\bibitem{B} E. Blasten, Nonradiating sources and transmission eigenfunctions vanish at corners and edges, SIAM J. Math. Anal., 6 (2018):  6255--6270.

%\bibitem{BC}N. Bleistein and J. Cohen, Nonuniqueness  in  the  inverse  source  problem  in  acoustics  and electromagnetics, J. Math. Phys., 18 (1977): 194--201.


\bibitem{CGMS2020} B. Chen, Y. Guo, F. Ma and Y. Sun, Numerical schemes to reconstruct three-dimensional time-dependent point sources of acoustic waves, Inverse Problems, 36 (2020): 075009.


%\bibitem{CB2008} M. Cheney and B. Borden,   Imaging moving targets from scattered waves, Inverse Problems 24 (2008): 035005.

\bibitem{CIL} J. Cheng, V. Isakov and S. Lu,  Increasing stability in the inverse source problem with many frequencies, J. Differential Equations, 260 (2016): 4786-4804.

%\bibitem{Co79}J. Cooper, Scattering of plane waves by a moving obstacle. Arch. Ration. Mech. Anal. 71 (1979): 113-149.

\bibitem{CK1998}D. Colton and R. Kress,  Inverse Acoustic and Electromagnetic Scattering Theory, 4th edition, Springer, Berlin, 2019.

%\bibitem{CoStra} J. Cooper and W. Strauss, Scattering of waves by periodically moving bodies. J. Funct. Anal. 47 (1982): 180-229.

%\bibitem{DW1973} A. J. Devaney and E. Wolf,  Radiating and non-radiating classical current distributions and the fields they generate,  Phys. Rev. D,  8 (1973): 1044-7.

%\bibitem{DS1982} A. J. Devaney and G. Sherman, Nonuniqueness in inverse source and scattering problems, IEEE Trans. Antennas Propag., 30 (1982): 1034-1037.


%\bibitem{EN2011} A. El-Badia, T. Nara
%An inverse source problem for Helmholtz's equation from the Cauchy data with a single wave number,
%Inverse Problems, 27 (2011): 105001.


\bibitem{EV09} M. Eller and N. Valdivia, Acoustic source identification using multiple frequency information, Inverse Problems, 25 (2009): 115005.


%\bibitem{FSY} A. C. Fannjiang, T. Strohmer and P. Yan, Compressed remote sensing of sparse objects, SIAM J. Imag. Sci., 3 (2010): 595-618.


\bibitem{FGPT2017} J. Fournier, J. Garnier, G. Papanicolaou and C. Tsogka,  Matched-filter and correlation-based imaging for fast moving objects using a sparse network of receivers, SIAM J. Imag. Sci., 10 (2017): 2165-2216.


\bibitem{GF2015} J. Garnier and M. Fink, Super-resolution in time-reversal focusing on a moving source, Wave Motion, 53 (2015): 80-93.

\bibitem{GS} R. Griesmaier and C. Schmiedecke, A Factorization method for multifrequency inverse source problem with sparse far-field measurements, SIAM J. Imag. Sci., 10 (2017): 2119-2139.

%\bibitem{G11} R. Griesmaier, Multi-frequency orthogonality sampling for inverse obstacle scattering problems, Inverse Problems, 27 (2011):  085005.

%\bibitem{GH15} R. Griesmaier and M. Hanke, Multifrequency impedance imaging with multiple signal classification, SIAM J. Imag. Sci., 8 (2015): 939-967.

%\bibitem{GHT12} R. Griesmaier, M. Hanke, and T. Raasch, Inverse source problems for the Helmholtz equation and the windowed Fourier transfrom, SIAM J. Sci. Comput., 34 (2012): A1544-A1562.

%\bibitem{GHT13}R. Griesmaier, M. Hanke, and T. Raasch, Inverse source problems for the Helmholtz equation and
%the windowed Fourier transform II, SIAM J. Sci. Comput., 35 (2013): A2188-A2206.

%\bibitem{GHS}R. Griesmaier, M. Hanke and J. Sylvester, Far field splitting for the Helmholtz equation, SIAM J. Numer. Anal., 52 (2014): 343-362.


\bibitem{GS17} R. Griesmaier and C. Schmiedecke, A multifrequency MUSIC algorithm for locating small inhomogeneities in inverse scattering, Inverse Problems, 33 (2017): 035015.

\bibitem{GGH2022} H. Guo and G. Hu, Inverse wave-number-dependent source problems for the Helmholtz equation, SIAM J. Numer. Anal., 62 (2024): 1372-1393.

\bibitem{GHM2023} H. Guo, G. Hu and G. Ma, Imaging a moving point source from multi-frequency data measured at one and sparse observation directions (part I): far-field case, SIAM J. Imag. Sci., 16 (2023): 1535-1571.

\bibitem{GHM2} H. Guo, G. Hu and G. Ma, Inverse wave-number-dependent source problems for the Helmholtz equation with partial information on radiating period, arXiv: 2401.07193.


%\bibitem{GHZ22} H. Guo, G. Hu and M. Zhao,
%Direct sampling method to inverse wave-number-dependent source problems: determination of the support of a stationary source, Inverse Problems, 39 (2023): 105008.



\bibitem{HKLZ2019} G. Hu, Y. Kian, P. Li and Y. Zhao, Inverse moving source problems in electrodynamics, Inverse Problems, 35 (2019): 075001.

\bibitem{HKZ2020} G. Hu, Y. Kian and Y. Zhao, Uniqueness to some inverse source problems for the wave equation in unbounded domains, Acta Mathematicae Applicatae Sinica, English Series, 36 (2020): 134-150.

%\bibitem{HL2020} G. Hu and J. Li, Uniqueness to inverse source problems in an inhomogeneous medium with a single far-field pattern, SIAM J. Math. Anal., 52 (2020): 5213-5231.


\bibitem{HLY20} G. Hu, Y. Liu  and M. Yamamoto, Inverse moving source problem for fractional diffusion(-wave) equations: Determination of orbits, Inverse Problems and Related Topics ed J Cheng, S Lu and M Yamamoto (Singapore: Springer) pp. 81-100, 2020.



%\bibitem{Ik99} M. Ikehata, Reconstruction of a source domain from the Cauchy data, Inverse Problems, 15
%(1999): 637--645.

\bibitem{T2022}H. A. Jebawy, A. Elbadia and F. Triki, Inverse moving point source problem for the wave equation, Inverse Problems, 38 (2022): 125003.

%\bibitem{JLZ19} X.  Ji, X. Liu and B. Zhang, Inverse acoustic scattering with phaseless far field data: uniqueness, phase retrieval, and direct sampling methods, SIAM J. Imag. Sci., 12 (2019):  1163-1189.

%\bibitem{JLZ}X.  Ji, X. Liu and B. Zhang, Target reconstruction with a reference point scatterer using phaseless far field patterns, SIAM J. Imag. Sci., 12 (2019): 372-391.

\bibitem{JLZ2019} X.  Ji, X. Liu and B. Zhang,  Phaseless inverse source scattering problem: phase retrieval, uniqueness and direct sampling methods, J. Comput. Phys. X, 1 (2019): 100003.

\bibitem{K98} A. Kirsch, Characterization of the shape of a scattering obstacle using the spectral data of the far field
operator, Inverse Problems, 14 (1998):  1489-1512.

%%\bibitem{K2002} A. Kirsch, The MUSIC algorithm and the factorization method in inverse scattering theory for inhomogeneous media, Inverse Problems, 18 (2002): 1025-1040.

\bibitem{KG08} A. Kirsch and N. Grinberg, {\rm The Factorization Method for Inverse Problems}, Oxford University Press,
Oxford, UK, 2008.

%\bibitem{KS03} S. Kusiak and J. Sylvester, The scattering support, Comm. Pure Appl. Math., 56 (2003): 1525-1548.

\bibitem{LY} P. Li and G. Yuan, Increasing stability for the inverse source scattering problem with multi-frequencies, Inverse Problems and Imaging, 11 (2017): 745-759.



\bibitem{LGS2021} Y. Liu, Y. Guo, and J. Sun, A deterministic-statistical approach to reconstruct moving sources using sparse partial data, Inverse Problems, 37 (2021): 065005.

\bibitem{LGY21} Y. Liu, G. Hu and M. Yamamoto, Inverse moving source problem for time-fractional evolution equations: determination of profiles, Inverse Problems, 37 (2021): 084001.

\bibitem{LMZ}
X. Liu, S. Meng and B. Zhang, Modified sampling method with near field measurements, SIAM J. Appl. Math., 82 (2022): 244-266.

\bibitem{Liu2021}Y. Liu, Numerical schemes for reconstructing profiles of moving sources in (time-fractional) evolution equations, RIMS Kokyuroku, 2174 (2021): 73-87.

%\bibitem{MH22} G. Ma and G. Hu,  Factorization method with one plane wave: from model-driven and data-driven perspectives, Inverse Problems, 38 (2022):  015003.

\bibitem{MGH2023} G. Ma, H. Guo and G. Hu, Imaging a moving point source from multi-frequency data measured at one and sparse observation points (part II): near-field case in 3D, SIAM J. Imag. Sci., 17 (2024), 1377-1414.


\bibitem{NIO2012} E. Nakaguchi, H. Inui and K. Ohnaka, An algebraic reconstruction of a moving point source for a scalar wave equation, Inverse Problems, 28 (2012): 065018.

\bibitem{Ohe2011} T. Ohe, H. Inui and K. Ohnaka,  Real-time reconstruction of time-varying point sources in a three-dimensional scalar wave equation, Inverse Problems, 27 (2011): 115011.


%\bibitem{P10} R. Potthast, A study on orthogonality sampling, Inverse Problems, 26 (2010): 074015.

%\bibitem{Sy06} J. Sylvester, Notions of support for far fields, Inverse Problems, 22 (2006): 1273--1288.


%\bibitem{SK05} J. Sylvester and J. Kelly, A scattering support for broadband sparse far field measurements, Inverse
%Problems, 21 (2005): 759-771.






\bibitem{T2020} O. Takashi, Real-time reconstruction of moving point/dipole wave sources from boundary measurements, Inverse Probl. Sci. Eng., 28 (2020): 1057-1102.

\bibitem{WKT2022}S. Wang, M. Karamehmedovic and F. Triki, Localization of moving sources: uniqueness, stability and Bayesian inference, SIAM Journal on Applied Mathematics, 83 (2023): 1049-1073.

\bibitem{ZG} D. Zhang and Y. Guo,  Fourier method for solving the multi-frequency inverse source problem for the Helmholtz equation, Inverse Problems, 31 (2015): 035007.


%%%%%%%%%%%%%%%%%%%%%%%%%%%%%%%%%%%%%



































\end{thebibliography}
\end{document}